\documentclass[11pt,a4paper,oneside]{article}

\usepackage{emptypage}
\usepackage{paralist}
\usepackage[english]{babel} 
\usepackage[utf8]{inputenc}
\usepackage{amsfonts}
\usepackage{amscd}
\usepackage{amssymb}
\usepackage{amsmath}
\usepackage{hyperref}
\hypersetup{
	colorlinks=true,
	linkcolor=blue,
	filecolor=magenta,      
	urlcolor=cyan,
}
\usepackage{cleveref}
\usepackage{bm,bbm}
\usepackage{amsbsy}
\usepackage{mathrsfs}
\usepackage{fancyhdr}
\usepackage{amsthm} 
\usepackage{latexsym}
\usepackage{makeidx} 
\usepackage{graphicx} 
\usepackage{mathtools}
\usepackage{geometry}
\usepackage{tikz}
\DeclarePairedDelimiter\abs{\lvert}{\rvert}
\DeclarePairedDelimiter\norm{\lVert}{\rVert}
\makeatletter
\let\oldabs\abs
\def\abs{\@ifstar{\oldabs}{\oldabs*}}
\let\oldnorm\norm
\def\norm{\@ifstar{\oldnorm}{\oldnorm*}}
\makeatother

\theoremstyle{plain}
\newtheorem{thm}{Theorem}[section]
\newtheorem{thm*}{Theorem}
\newtheorem*{thmA*}{Theorem A}
\newtheorem{prop}[thm]{Proposition}
\newtheorem{prop*}{Proposition}
\newtheorem{lemma}[thm]{Lemma}

\theoremstyle{definition}
\newtheorem{remark}[thm]{Remark}
\theoremstyle{definition}

\theoremstyle{definition}

\theoremstyle{definition}
\newtheorem{defin}[thm]{Definition}
\theoremstyle{definition}
\newtheorem*{notation}{Notation}
\newtheorem*{acknowledgements}{Acknowledgements}

\newtheorem*{claim*}{Claim}
\newtheorem*{claim1*}{Claim 1}
\newtheorem*{claim2*}{Claim 2}
\newtheorem*{claim3*}{Claim 3}
\newtheorem*{claim4*}{Claim 4}
\newtheorem*{fact*}{Fact}
\newtheorem*{fact1*}{Fact 1}
\newtheorem*{fact2*}{Fact 2}
\newtheorem*{fact3*}{Fact 3}
\newtheorem*{fact4*}{Fact 4}
\newtheorem*{assumption*}{Assumption}
\newtheorem*{assumption1*}{Assumption 1}
\newtheorem*{assumption2*}{Assumption 2}
\newtheorem*{assumption3*}{Assumption 3}
\newtheorem*{assumption4*}{Assumption 4}
\newtheorem*{step1*}{Step 1}
\newtheorem*{step2*}{Step 2}
\newtheorem*{step3*}{Step 3}
\newtheorem*{remark*}{Remark}

\newcommand{\sphere}{\mathbb{S}}

\newcommand{\R}{\mathbb{R}}
\newcommand{\N}{\mathbb{N}}

\newcommand{\normps}[3]{\norm{#1}_{L^{#2}(#3)}}

\newcommand{\dv}{\,d V_g}
\newcommand{\dg}{\Delta_g}
\newcommand{\nabg}{\nabla_g}
\newcommand{\obig}{\mathcal{O}}

\newcommand{\ssubset}{\subset\joinrel\subset}

\DeclareMathOperator\supp{supp}

\renewcommand{\epsilon}{\varepsilon}
\renewcommand{\theta}{\vartheta}
\let\temp\phi
\let\phi\varphi
\let\varphi\temp

\title{Compactness of Palais-Smale sequences with controlled Morse Index for a Liouville type functional}
\author{Francesco Malizia \thanks{Scuola Normale Superiore, Piazza dei Cavalieri 7, 56126 Pisa. e-mail: francesco.malizia@sns.it}}

\date{}

\begin{document}

	\maketitle

	{\footnotesize
		\begin{abstract}
			
			\noindent We prove that Palais-Smale sequences for Liouville type functionals on closed surfaces are precompact whenever they satisfy a bound on their Morse index. As a byproduct, we obtain a new proof of existence of solutions for Liouville type mean-field equations in a supercritical regime. Moreover, we also discuss an extension of this result to the case of singular Liouville equations.
			
			\vspace{3ex}
			

			
	\end{abstract}}

\section{Introduction}\label{sectionintro}

Let $(M,g)$ be a closed surface of unitary volume. We are interested in the following equation:

\begin{equation}\label{meanfield}
	-\dg u =\lambda\left(\frac{he^u}{\int_M he^u \dv}-1\right),
\end{equation}
where $h>0$ is a smooth function, $\lambda>0$ is a real number and $\dg$ is the Laplace-Beltrami operator. Equations like \eqref{meanfield} are usually referred to as \emph{Liouville type} mean-field equations, and are motivated for instance by problems in conformal geometry and mathematical physics.

In conformal geometry, a well-known problem, whose study was initiated by Kazdan and Warner in \cite{kazdanwarner-Annals-1974}, is that of conformally prescribing the Gaussian curvature. More precisely, given $(M,g)$ as above, if $\tilde{g}:=e^ug$ is a new metric conformal to $g$ with conformal factor $e^u$, then the function $u$ solves 
\begin{equation}\label{gausscurvatureeq}
	-\dg u+ 2K_g u=2K_{\tilde{g}} e^u \qquad \text{on $M$,}
\end{equation}
where $K_g$ and $K_{\tilde{g}}$ denote the Gaussian curvatures of $g$ and $\tilde{g}$ respectively. By the classical Uniformization Theorem, there always exists a metric $\tilde{g}$ conformal to $g$ whose Gaussian curvature is constant. More generally, one could try to identify all the functions $K$ on $M$ which can be realised as the Gaussian curvature of some metric conformal to $g$; this amounts to find all the functions $K$ for which equation \eqref{gausscurvatureeq} is solvable with $K$ in place of $K_{\tilde{g}}$. There are some obstructions to the existence of solutions for \eqref{gausscurvatureeq}, the simplest one being Gauss-Bonnet's formula, and the case $(M,g)=(\sphere^2,g_{\sphere^2})$, referred to as \emph{Nirenberg's problem}, represents the most delicate and thoroughly studied situation, see e.g. \cite[Chapter 6]{aubinbook} and references therein.

In mathematical physics, equation \eqref{meanfield} arises when one employs methods of classical statistical mechanics to describe large isolated vortices, cfr. \cite{cagliotilionsmarchioropulvirenti1992}, \cite{cagliotilionsmarchioropulvirenti1995} (here one actually gets the counterpart of \eqref{meanfield} on bounded domains with Dirichlet boundary conditions), or in selfdual Gauge field theories, see \cite{tarantellobook} as well as its references.

\bigskip

 Solutions of \eqref{meanfield} are critical points of the functional

\begin{equation}\label{functional}
	J_\lambda (u)= \frac{1}{2}\int_M \abs{\nabg u}^2 \dv +\lambda \int_M u\dv -\lambda \log\left(\int_M h e^u \dv\right),
\end{equation} 
for $u\in H^1(M)$, which is well-defined and smooth by virtue of the Moser-Trudinger inequality \eqref{moser-trudinger-inequality}. It is known that equation \eqref{meanfield} is solvable whenever $\lambda\notin 8\pi\N$: aside from the easy case $\lambda\in(0,8\pi)$ (where the aforementioned Moser-Trudinger inequality implies the coerciveness of $J_\lambda$, in which case one can minimize it), the existence of a solution for $\lambda>8\pi$ was proved using more sophisticated variational arguments together with the so-called ``Struwe's monotonicity trick" and blow-up analysis for sequences of solutions, see \cite{tarantellostruwe}, \cite{dingjostliwang-AIHP-1999}, \cite{malchiodi2008topologicalmethods}, \cite{djadli2008allgenus}. 

However, little is known about the behaviour of Palais-Smale sequences for $J_\lambda$. In this paper we address this last issue and show that, under an additional and natural second order assumption on such sequences, it is actually possible to prove that they subconverge to a critical point of \eqref{functional}.

As a consequence, we obtain a new proof of the existence of solutions for \eqref{meanfield} when $\lambda>8\pi$, $\lambda\notin 8\pi\N$, which does not rely upon Struwe's trick.

\medskip

Let $(u_n)_n$ be a Palais-Smale (PS) sequence for $J_\lambda$, namely $J_\lambda(u_n)\to c\in \R$ and $J'_\lambda(u_n)\to 0$ in $H^{-1}(M)$ as $n$ goes to infinity; here $J_\lambda'$ denotes the differential of the functional $J_\lambda$, and later on we will denote by $J_\lambda''$ its second order differential.
	If we add a constant $C$ to $u_n$, then $J_\lambda(u_n+C)=J_\lambda(u_n)$ and $J'_\lambda(u_n+C)=J'_\lambda(u_n)$; as a consequence, from now on we will renormalize $(u_n)_n$ by asking
	\begin{equation}\label{normalization}
		\int_M h e^{u_n} \dv=1 \qquad \forall n\in \N.
	\end{equation}

Given such a sequence $(u_n)_n$, we now have two possibilities: either there exists a subsequence which is bounded in $H^1(M)$, or else the whole sequence is unbounded. 

In the first case, one can easily prove the existence of a subsequence weakly converging in $H^1(M)$ to a critical point of \eqref{functional}, which is a weak (a posteriori classical) solution of \eqref{meanfield}, see e.g. \cite{malchiodi2008topologicalmethods} for details.

On the other hand, when $(u_n)_n$ is unbounded it is usually difficult to describe its behaviour. However, if we further assume $(u_n)_n$ to satisfy an additional bound on the Morse index, then we are able to characterize the loss of compactness.

The condition is the following: there exists $N\in\N$ such that, $\forall n\in\N$,
\begin{equation}\label{morseindexproperty}
	 \sup\left\{\dim(E)\mid E \text{ subspace of $H^1(M)$ s.t. } J_\lambda^{''}(u_n)[w,w]<-\frac{1}{n}\norm{w}^2 \,\,\, \forall w\in E\right\}\leq N.
\end{equation}
In particular, this condition implies that, whenever $(u_n)_n$ admits a subsequence converging to a limit $u$ in $H^1(M)$, then the Morse index of $J_\lambda$ at $u$ is at most $N$. Some comments on the naturality of assumption \eqref{morseindexproperty} are given later on in the introduction.

With the help of \eqref{morseindexproperty}, it is possible to carry out a blow-up analysis for our sequence $(u_n)_n$ and thus prove the main result of this paper:

 \begin{thm}\label{maintheorem}
	Given $\lambda\in (8\pi k,8\pi(k+1))$ and $k\in \N$, let $(u_n)_n$ be a Palais-Smale sequence for $J_\lambda$ which satisfies \eqref{normalization} and \eqref{morseindexproperty} for a suitable $N=N(k)$. Then $(u_n)_n$ is bounded in $H^1(M)$; in particular, it subconverges to a classical solution of \eqref{meanfield}.
\end{thm}

The idea of the proof is to argue by contradiction and show that, if we assume $(u_n)_n$ to be unbounded, then the total ``mass" $\lambda \int_Mhe^{u_n}\dv$ (and thus $\lambda$ in view of \eqref{normalization}) should converge, up to a subsequence, to an integer multiple of $8\pi$, which is forbidden by our assumptions. Before entering more into details, we quickly describe the blow-up behaviour of a sequence of \emph{exact} solutions of \eqref{meanfield} in order to highlight the main differences and additional difficulties related to the analysis of Palais-Smale sequences.

Given a sequence of solutions for \eqref{meanfield}, either it is uniformly bounded in $L^\infty$ (and thus also in $C^{2,\alpha}$ by standard elliptic regularity), or else it blows up. In the latter case, we can suitably rescale around local maxima in order to find subsequences of functions converging to the solution of a limit equation in $\R^2$, which turns out to be the \emph{Liouville equation} $-\Delta u= e^u$. Moreover, the scaling invariance of the ``mass density" $he^{u_n}\dv$ implies that the limit function satisfies $\int_{\R^2} e^u<+\infty$. At this point, a classification result of Chen and Li \cite{chenli-Duke-1991-classification} gives the quantisation $\int_{\R^2}e^u =8\pi$. Moreover, it was proved in \cite{yyli-CMP-1999-harnacktype} (see also \cite{bartolucci-tarantello-2002-CMP} for a simpler proof relying upon Pohozaev identity) that multiple ``bubbles" (i.e. concentration profiles) cannot accumulate in the same point and that there is no additional mass outside that of the bubbles; also, $u_n\to-\infty$ uniformly away from the concentration points. All together, these facts show that $\lambda\in8\pi\N$ necessarily.

When dealing instead with Palais-Smale sequences, we face some additional difficulties. To begin, the low regularity of $(u_n)_n$ requires us to describe the concentration regions in an integral way. Also, $(u_n)_n$ \underline{does not solve any equation}, so, in order to show that the concentration profiles do carry a quantised amount of mass, we need to define and work with auxiliary sequences of more regular functions which are solutions of suitable PDEs and which are close to $u_n$ in $H^1$-norm. However, since the $L^2$-norm is not preserved by scaling, we must also pay attention and renormalise when scaling these auxiliary sequences in order to get in the limit a solution for the Liouville equation in $\R^2$.
After extracting a finite number of concentration profiles, we show that these ones do not carry additional mass in their ``necks": this is the most delicate point in the blow-up analysis and the Morse index property \eqref{morseindexproperty} comes here into play. Finally, we show that there is no mass everywhere else; differently from the case of (exact) solutions, here it seems that we cannot exclude the \emph{clustering} of bubbles, that is, multiple bubbles might accumulate at the same point in the limit (this can also occur with sequences of exact solutions in bounded domains if we have no control on the oscillations at the boundary, see \cite{chen-1999-CAG}; instead, it cannot happen on closed manifolds as the Green representation formula implies that we always have uniformly bounded oscillations away from blow-up points, see \cite{bartolucci-tarantello-2002-CMP}). As in \cite{lishafrir-1994}, we thus employ a concentration-compactness argument on clusters of bubbles to show that, even in this case, there is no additional mass. Finally, with a global estimate relying upon Green's function, we rule out the occurrence of mass away from concentration points: we thus obtain that $\lambda\in 8\pi\N$ necessarily, which gives the desired contradiction.

\bigskip

A more general class of problems is that of \emph{singular} Liouville type mean-field equations, that is, equations of the form
\begin{equation*}\label{singularintroliouville}
	-\dg \bar{u}= \lambda\left(\frac{\bar{h}e^{\bar{u}}}{\int_M \bar{h}e^{\bar{u}} \dv}-1\right)-4\pi\sum_{i=1}^m\alpha_i (\delta_{q_i}-1),
\end{equation*}
where $m\in \N$, $\alpha_i>-1$ $\forall i$ and $\delta_{q_i}$ denotes the Dirac mass centered at $q_i\in M$.

These equations are related to the problem of prescribing the Gaussian curvature where we allow the conformal metric to have \emph{conical singularities}, and to the theory of selfdual Chern-Simons vortices; we refer to the survey \cite{tarantello-2010-DCDS-liouvillesurvey} and again to \cite{tarantellobook} for further informations.

 If we let $u(x):=\bar{u}(x)+4\pi\sum_{i=1}^m\alpha_i G(x,q_i)$, where $G$ denotes the Green's function of Laplacian defined in \eqref{greenfunction}, then $u$ is a solution for
\begin{equation}\label{singularequation}
	-\dg u =\lambda\left(\frac{\tilde{h}e^u}{\int_M \tilde{h}e^u \dv}-1\right),
\end{equation}
where now $\tilde{h}(x):=\bar{h}(x)e^{-4\pi\sum_{i=1}^m\alpha_i G(x,q_i)}$ is a potential function with zeros in the points $q_i$ with $\alpha_i>0$ and poles in the points $q_i$ with $\alpha_i\in(-1,0)$.

Equation \eqref{singularequation} is the Euler-Lagrange equation of the corresponding integral functional \eqref{functional} with $h$ replaced by $\tilde{h}$. It turns out that the blow-up analysis of PS-sequences $(u_n)_n$ satisfying \eqref{morseindexproperty} in this case is more involved and the procedure employed in this paper only works in the particular case in which $\alpha_i\leq 1$ $\forall i=1,\dots,m$, see Theorem \ref{singulartheorem}. However, since the corresponding analysis for sequences of \emph{exact} solutions carried out in \cite{bartolucci-tarantello-2002-CMP}, \cite{bartolucci-montefusco-2007} only requires the assumption $\alpha_i>-1$, we expect that our result could be extended to cover this more general case as well.

\bigskip

Given the aforementioned results for granted, it is clear that the whole problem of solving \eqref{meanfield} or \eqref{singularequation} boils down to constructing a PS-sequence for the associated functional which verifies the Morse index property \eqref{morseindexproperty}. This is always possible whenever our functional posseses a topological structure which allows the use of a min/max scheme of some finite dimension:

\begin{thmA*}[\textup{\cite[Theorem 1.4]{fang-ghoussoub-CPAM-1994}}]
	Let $H$ be an Hilbert space and $I:H\to\R$ a functional of class $C^{2,\alpha}$. Let $D\subseteq \R^{N}$ be a compact subset, $B\subseteq D$, and define
	\begin{equation*}
		\Gamma:=\left\{\gamma\in C(D;H)\mid \gamma_{\mid B}=\gamma_0, \quad \text{where }\gamma_0\in C(B;H) \text{ is a given function}\right\}.
	\end{equation*}
	Let 
	\begin{equation}\label{minmaxproperty}
		c:=\inf_{\gamma\in\Gamma}\max_{x\in D}I(\gamma(x)), \quad \text{and assume} \quad c>c_0:=\max_{y\in B} I(\gamma_0(y)).
	\end{equation}
	There exist $(\gamma_n)_n\subseteq \Gamma$ and $(u_n)_n\subseteq H$ with $u_n=\gamma_n(x_n)$ for some $x_n\in D$, such that
	\begin{equation*}
		I(u_n)\longrightarrow c, \quad \quad \quad \norm{I'(u_n)}_{H^\ast}\longrightarrow0 \quad \text{as $n\to\infty$},
	\end{equation*}  
	and, $\forall n\in \N$,
	\begin{equation}\label{morseindexpropertygeneralfunctional}
		\sup\left\{\dim(E)\mid E \text{ subspace of $H$ s.t. } I^{''}(u_n)[w,w]<-\frac{1}{n}\norm{w}^2 \,\,\, \forall w\in E\right\}\leq N.
	\end{equation}
\end{thmA*}

\bigskip

 In other words, whenever we are in condition to apply a min/max scheme of dimension $N$ to the functional $I$, we then obtain a Palais-Smale sequence with the above additional property on the Morse index. In particular, Theorem A applies to our functional $J_\lambda$ in \eqref{functional} when $\lambda\notin 8\pi\N$; we now provide a brief description of the variational setting and refer the reader to \cite{malchiodi2008topologicalmethods} for further details.  
  
 Let $\lambda\in (8\pi k, 8\pi(k+1))$, $1\leq k\in \N$, and define
 \begin{equation*}
 	M_k:=\left\{\sum_{i=1}^k t_i\delta_{q_i}\mid\sum_{i=1}^k t_i=1, \,\, t_i\geq 0 \,\, \forall i, \,\, q_i\in M, \,\, \text{$\delta_{q_i}$ Dirac mass centered at $q_i\in M$}\right\},
 \end{equation*}
 endowed with the weak topology of distributions. This is called \emph{formal space of $k$-barycenters on $M$}; it is a stratified space, that is, finite union of open manifolds (i.e. noncompact without boundary) of different dimensions, whose maximal one is $3k-1$. In particular, $M_k$ has the structure of a CW-complex of dimension $3k-1$, see  e.g. the appendix of \cite{battaglia-jevnikar-malchiodi-ruiz-Advances-2015}. We can now use $M_k$ to build our min-max scheme: let $\Sigma:=\frac{M_k\times[0,1]}{\sim}$ be the contractible cone over $M_k$ (i.e. $M_k\times\{1\}$ is collapsed into a single point). Being $\Sigma$ itself a CW-complex, we can embed it as a compact subset of $\R^N$ for a big enough $N\in \N$ through a map $\Psi:\Sigma\to\R^N$. We then take $D:=\Psi(\Sigma)$ and $B:=\Psi(M_k\times\{0\})$ as the sets $D$ and $B$ in the statement of Theorem A. It is now possible to choose a suitable map $\gamma_0\in C(B;H^1(M))$ such that \eqref{minmaxproperty} holds: we refer again to \cite{malchiodi2008topologicalmethods} for details. As a consequence, we can apply Theorem A to get the existence of a (PS)-sequence $(u_n)_n\subseteq H^1(M)$ for $J_\lambda$ satisfying \eqref{morseindexproperty}; moreover, we can freely assume that $(u_n)_n$ satisfies the normalization \eqref{normalization}.

Regarding instead singular Liouville equations, in many cases (which depend upon the genus of $M$, the value of $\lambda$ and the $\alpha_i$'s), we are still in condition to apply Theorem A to our functional, obtaining a Palais-Smale sequence $(u_n)_n$ with the second-order information \eqref{morseindexproperty}. For a precise description of the variational structure of the singular Liouville functional, we refer the reader to e.g. \cite{carlotto-malchiodi-2012-JFA} and \cite{bartolucci-demarchis-malchiodi-2011-IMRN}.

\subsection*{Liouville equations versus harmonic maps and some open questions}

In this last part of the introduction we want to briefly sketch a comparative picture between Liouville equations and harmonic maps in order to highlight the main differences and similarities between them and formulate some questions.

Given a closed Riemann surface $(M,g)$ and a target (closed) manifold $(N^n,h)$ of dimension $n\geq3$, which we assume to be isometrically embedded in $\R^k$ for some suitable $k\in\N$ (this is always possible due to Nash's Theorem), we define the \emph{Dirichlet energy}

\begin{equation*}
	E(f):=\frac{1}{2}\int_M\abs{\nabla f}^2\dv
\end{equation*}
for maps $f\in W^{1,2}(M;N)$.
Critical points (in a suitable sense) of $E$ are called \emph{harmonic maps} and they satisfy 
\begin{equation*}
	-\Delta f = A(f)(\nabla f, \nabla f),
\end{equation*}
where $A(f)$ denotes the second fundamental form of $N\subset\R^k$ computed along the image of $f$; see e.g. \cite{heleinbook-harmonicmaps} for precise statements and further details.

The study of existence of harmonic maps is a major topic in geometric analysis due to their link with minimal immersions: if an harmonic map is also a conformal immersion, then its image is a minimal immersed surface (\cite{eelles-sampson-1964-americanj}). The conformal invariance of the problem implies that we have to deal with the occurence of concentration phenomena when studying the behaviour of sequences of (approximate) solutions: indeed, the strategy of scaling back to isolate bubbles was employed for the first time ever by Sacks and Uhlenbeck in \cite{sacks-uhlenbeck-1981-Annals} for a ``perturbation" of $E$.

\medskip

As for Liouville equations, we have a complete picture of the blow-up behaviour of sequences of harmonic maps with uniformly bounded energy, see e.g. \cite[Theorem 2.2]{parker-1996-JDG}. Long story short, there exists a subsequence of harmonic maps ``bubble converging" to a limit harmonic map plus a finite number of bubbles, i.e. harmonic maps from $\sphere^2$ into $N$; moreover, we have no energy loss and all the necks have zero lenght in the limit. 

On the other hand, this result no longer holds when we consider instead \emph{Palais-Smale sequences} for the harmonic map energy $E$; this is due to the fact that, in this case, we may lose any positive value of energy in the necks, see \cite[Propositions 4.1 and 4.2]{parker-1996-JDG}. Hence $E$ \emph{does not} satisfy the Palais-Smale condition.

The natural counterpart of this fact for Liouville equations is the following question: 

\medskip

\noindent{\bf{Question:}} does the Palais-Smale property hold for $J_\lambda$ when $\lambda\notin 8\pi\N$?

\medskip

By Theorem \ref{maintheorem}, the answer is positive  for sequences which satisfy the additional condition \eqref{morseindexproperty}, but we were unable to remove \eqref{morseindexproperty} or, conversely, to provide an example of a noncompact (PS)-sequence at energy level $\lambda\notin 8\pi\N$, $\lambda>8\pi$: this seems to be a challenging problem.

We also mention that the Palais-Smale property \emph{does not} hold in general also for the Moser-Trudinger functional, see \cite{costa-tintarev-2014-JFA}. This failure for both harmonic maps and Moser-Trudinger functional seems to suggest that the same conclusion might also be true for the Liouville functional $J_\lambda$, but, differently from these two other problems, the Liouville nonlinearity $e^u$ is not the ``most critical" one in dimension $2$.

\medskip

Coming back to harmonic maps, the lack of a bubble convergence theorem for general Palais-Smale sequences suggests to ask the following: 

\medskip

\noindent{\bf Question:} is it possible to prove a bubble convergence theorem for Palais-Smale sequences for the harmonic map energy $E$ under an additional assumption in the spirit of \eqref{morseindexpropertygeneralfunctional}?

\medskip

A positive answer to this question would allow, at least in theory (we would also need a counterpart of Theorem A in this new setting), a \emph{direct} use of min/max methods on $E$ in order to prove existence of harmonic maps. However, one would first need to find a suitable definition of ``approximate" Morse index for Palais-Smale sequences in this setting. As of today, most of the techniques employed to prove existence for a general target $N$ rely either on a suitable approximation of the energy functional $E$ by functionals satisfying the (PS)-condition in the spirit of \cite{sacks-uhlenbeck-1981-Annals} or \cite{lamm-2006-CalcVar}, or on flow methods as in \cite{struwe-1985-CMH-harmonicmapsheatflow}. 

In the latter case, one gets from the flow the existence of a (PS)-sequence $(f_n)_n$ in $W^{2,2}(M;N)$ with \emph{tension fields} $\tau_n:=\Delta f_n +A(f_n)(\nabla f_n, \nabla f_n)$ uniformly bounded in $L^2(M;N)$, for which the bubble tree convergence holds, see e.g. \cite{qing-tian-1997-CPAM}.

In the former case instead, while the energy quantization was proved to hold in \cite{lamm-2010-transactions-energyidentity} under an additional and natural ``entropy-type" assumption, this is still not true in general and one might also have necks collapsing to (nontrivial) geodesics in the limit, see \cite{li-wang-2015-PacificJ-nonquantisedalpha}.

\bigskip

We now briefly describe the structure of the paper: after recalling some basic facts in Section \ref{sectionconccompoactness}, in Section \ref{sectionbubbleextr} we describe how to isolate concentration profiles and in Section \ref{sectionsimpleblowup} we show that there is no further mass along their ``necks". At this point, in Section \ref{sectionglobalquantization} we prove a global quantization result and complete the proof of Theorem \ref{maintheorem}. Lastly, Section \ref{sectionsingulareq} briefly explains how to extend the results of the other Sections to the singular equation \eqref{singularequation}.

 \begin{notation}
 	Through the rest of this paper, $(M,g)$ will be a closed surface with unitary Riemannian volume. For any two points $x,y\in M$, $d_g(x,y)$ will denote their distance. Given a function $f:M\to \R$, we will denote with $\bar{f}:=\frac{1}{Vol_g(M)}\int_M f\dv$ its average value over $M$. We will always denote by $B_r(x)$ the metric ball of center $x$ and radius $r$.
 	
 	We will use $C$ to denote various positive constants which do not depend upon the index $n\in \N$ of the considered sequences; these constants are allowed to vary from line to line. Sometimes we will use subscripts to emphasize the dependence of $C$ with respect to some specific parameters, e.g. $C_R$ if $C$ depends upon $R$. 
 \end{notation}

\section{Concentration-compactness}\label{sectionconccompoactness}

From now on, $(u_n)_n$ will denote a Palais-Smale sequence for $J_\lambda$ satisfying \eqref{normalization}. To begin, we define an auxiliary sequence $(v_n)_n$ associated to $(u_n)_n$ and prove that those two sequences are ``close" in $H^1$-norm. 

For any $n\in \N$, let $v_n$ be the solution for

\begin{equation}\label{vn}
	\begin{cases*}
		-\dg v_n=\lambda\big(he^{u_n}-1\big) \qquad \text{on $M$,} \\
		\bar{v}_n=\bar{u}_n.
	\end{cases*}	
\end{equation}
We begin to recall the following version of Moser-Trudinger's inequality on closed surfaces:
\begin{thm}[\cite{fontana-mosertrudinger-manifolds}]\label{moser-trudinger-theorem}
	Let $(M,g)$ be a closed surface; there exists a constant $C=C(M)$ such that
	\begin{equation}\label{moser-trudinger-inequality}
		\frac{1}{Vol_g(M)}\int_M e^{(v-\bar{v})}\dv\leq Ce^{\frac{1}{16\pi}\norm{\nabg v}_{L^2(M)}^2} \qquad \forall v\in H^1(M).
	\end{equation}
\end{thm}
By virtue of \eqref{moser-trudinger-inequality}, we get  $u_n\in L^{q}(M)$ $\forall q\geq 1$, so, by standard elliptic estimates, we see that $v_n\in W^{2,q}(M) \,\, \forall q\geq 1$, see Chapter 5 of \cite{tarantellobook} for further details.

We now have the following:

\begin{lemma}\label{approxlemma}
	Let $(u_n)_n$, $(v_n)_n$ be as above. Then
	\begin{equation*}
 \norm{u_n- v_n}_{H^1(M)}\longrightarrow0 \qquad \text{as $n\to\infty$}.
	\end{equation*}
\end{lemma}

\begin{proof}
	Being $(u_n)_n$ a (PS)-sequence, one has
	\begin{equation*}
		o_n(1)\norm{\phi}_{H^1}=J_\lambda'(u_n)[\phi]=\int_M\nabg u_n\nabg \phi \dv+\lambda\int_M \phi \dv -\lambda\int_M he^{u_n}\phi\dv \qquad\forall \phi\in H^1(M).
	\end{equation*}
	At the same time, \eqref{vn} implies
	\begin{equation*}
		\int_M\nabg v_n \nabg \phi \dv =\lambda\int_M he^{u_n} \phi \dv -\lambda \int_M\phi\dv \qquad\forall \phi\in H^1(M).
	\end{equation*}
	From these formulae, one gets
	\[
	\int_M \nabg(u_n-v_n)\nabg\phi \dv =o_n(\norm{\phi}_{H^1}) \qquad\forall \phi\in H^1(M),
	\]
	which, together with $\bar{u}_n=\bar{v}_n$ and Poincaré inequality, allows us to conclude.
\end{proof}

Next we recall some fundamental facts about the Green's function of the Laplacian on closed manifolds wich will be heavily employed in the rest of this paper.

\begin{thm}[\cite{aubinbook}, Chapter 4]
	Given $(M,g)$ closed surface, there exists the Green's function $G$ of $-\dg$, that is, the distributional solution of
	\begin{equation}\label{greenfunction}
		\begin{cases*}
		-\dg G(x,\cdot)=\delta_x -\frac{1}{Vol_g(M)}  & \text{in $M$}, \\
		\int_M G(x,y)\dv(y)=0 &\text{$\forall x \in M$}.
	\end{cases*}
	\end{equation}
	$G$ is symmetric, it is smooth on $M\times M$$\backslash\{(x,x)\mid x\in M\}$ and satisfies
	\begin{equation}\label{greenfunctiondecay}
		\abs{G(x,y)-\frac{1}{2\pi}\log\frac{1}{d_g(x,y)}}\leq C, \qquad \abs{\nabg G(x,y)}\leq C \frac{1}{d_g(x,y)}  \qquad \forall x\not =y \in M,
\end{equation}	
where $C>0$ is a suitable positive constant. Moreover, if $-\dg u=f$, $f\in L^1(M)$ and $\int_M f=0$, the following representation formula holds:
\begin{equation}\label{greenrepresentationformula}
	u(x)=\bar{u}+\int_M f(y) G(x,y)	\dv(y) \quad \forall x\in M. 
\end{equation}
\end{thm}
	
We now state the main result of this Section, which closely follows \cite[Proposition 3.1]{malchiodi-2006-CRELLE}.	
\begin{lemma}\label{minimalmass}
		Let $\zeta_n$ be a solution of $\,-\dg \zeta_n= f_n$ on $M$ with $f\in L^1(M)$, $\int_M \abs{f_n}\dv\leq C$ $\forall n\in\N$. Then, up to subsequences:
\begin{enumerate}
	\item either $\int_M e^{q(\zeta_n-\bar{\zeta}_n)}\dv\leq C$ for some $C>0$ and $q>1$; 
	\item or else there exist $x_1,\dots,x_l\in M$ such that, $\forall \, r>0$, $\forall i\in \{1,\dots,l\}$, there holds
	\begin{equation}\label{minmasseq}
		\liminf_{n\to\infty}\int_{B_r(x_i)}\abs{f_n}\dv\geq4\pi.
	\end{equation}
\end{enumerate}
\end{lemma}

\begin{proof}
	Assume that the second alternative does not hold, namely $\forall \, x\in M \quad \exists r_x>0$, $\delta_x>0$, such that, for $n$ large enough and up to subsequences, one has 
	\begin{equation}\label{minmasseqproof}
	\int_{B_{r_x}(x)}\abs{f_n}\dv\leq 4\pi-\delta_x.
	\end{equation}
	Being $M$ compact, we can cover it with a finite number $T$ of balls $B_j:=B_{\frac{1}{2}r_{x_j}}(x_j)$, $j=1,\dots,T$, and we can extract a subsequence such that, when $n$ large enough and $\delta:=\min_j \delta_{x_j}$,
	\[
	\int_{B_j}\abs{f_n}\dv\leq4\pi-\delta \qquad \forall j=1,\dots,T.
	\]
	Using the Green's representation formula \eqref{greenrepresentationformula}, write
	\begin{equation*}
		\zeta_n(x)=\bar{\zeta}_n +\int_M G(x,y)f_n(y)\dv(y).
	\end{equation*}
	Let $\tilde{B}_j:= B_{r_{x_j}}(x_j)$; one has
	\begin{equation*}
		\zeta_n(x)-\bar{\zeta}_n=\int_{M\backslash\tilde{B}_j} G(x,y) f_n(y)\dv(y)+\int_{\tilde{B}_j} G(x,y) f_n(y)\dv(y).
	\end{equation*}
	Pick $x\in B_j$; then $y\mapsto G(x,y)$ is bounded in $M\backslash\tilde{B}_j$ (the bound depends upon $M$ and $r_{x_j}$), so
	\begin{equation*}
		e^{(\zeta_n(x)-\bar{\zeta}_n)}\leq C \exp\left(\int_{\tilde{B}_j} G(x,y) f_n(y)\dv(y)\right).
	\end{equation*}
	We can now use Jensen's inequality and Fubini/Tonelli's Theorem to estimate
	\begin{equation*}
		\begin{split}
			\int_{B_j} e^{q(\zeta_n(x)-\bar{\zeta}_n)}\dv(x) &\leq C\int_{B_j}\exp\left(\int_{\tilde{B}_j} q\abs{G(x,y)}\abs{f_n(y)}\dv(y)\right)\dv(x)  \\
			&=C\int_{B_j}\exp\left(\int_{\tilde{B}_j}q\abs{G(x,y)}\normps{f_n}{1}{\tilde{B}_j}\frac{\abs{f_n(y)}}{\normps{f_n}{1}{\tilde{B}_j}}\dv(y)\right)\dv(x) \\
			&\leq C\int_{B_j}\left(\int_{\tilde{B}_j} e^{q\abs{G(x,y)}\normps{f_n}{1}{\tilde{B}_j}}\frac{\abs{f_n(y)}}{\normps{f_n}{1}{\tilde{B}_j}}\dv(y)\right)\dv(x)  \\
			&=C\int_{\tilde{B}_j}\left(\int_{B_j} e^{q\abs{G(x,y)}\normps{f_n}{1}{\tilde{B}_j}}\dv(x)\right)\frac{\abs{f_n(y)}}{\normps{f_n}{1}{\tilde{B}_j}}\dv(y) \\
			&\leq C\sup_{y\in M} \int_{B_j} e^{q\abs{G(x,y)}\normps{f_n}{1}{\tilde{B}_j}}\dv(x).
		\end{split}
	\end{equation*}
	
	Finally, recalling \eqref{greenfunctiondecay}, we get
	\begin{equation*}
		\int_{B_j} e^{q(\zeta_n(x)-\bar{\zeta}_n)}\dv(x)\leq C\int_M\left(\frac{1}{d_g(x,y)}\right)^{\frac{q\normps{f_n}{1}{\tilde{B}_j}}{2\pi}}\dv(x).
	\end{equation*}
	The last integral is finite  if and only if $ \frac{q\normps{f_n}{1}{\tilde{B}_j}}{2\pi}<2$, that is, if and only if $q\normps{f_n}{1}{\tilde{B}_j}<4\pi$, which is indeed true by \eqref{minmasseqproof} whenever $q>1$ is close enough to $1$.
	Hence $\int_{B_j}e^{q(\zeta_n-\bar{\zeta}_n)}\dv<+\infty$ and, being $B_1,\dots,B_T$ a finite cover for $M$, the first alternative holds.
\end{proof}

\begin{remark}\label{localminmasslemma}
	Using Green's representation formula in a ball $B_R$ and proceeding along the same lines of the above proof, it is possible to obtain the following localised version of Lemma \ref{minimalmass}, which will be heavily employed in the next Sections:
	let $\zeta_n\in W^{1,r}_0(B_R)$, $r\geq 1$, be a solution for $-\Delta \zeta_n=f_n$ in $B_R$, where $(f_n)_n$ is bounded in $L^1(B_R)$. Then, up to a subsequence:
	\begin{enumerate}
		\item either $\int_{B_{\frac{R}{2}}} e^{q\zeta_n}\,dx\leq C$ for some $C>0$ and $q>1$; 
		\item or else there exist $x_1,\dots,x_l\in \overline{B}_{\frac{R}{2}}$ such that, $\forall \, r>0$, $\forall i\in \{1,\dots,l\}$, there holds
		\begin{equation*}
			\liminf_{n\to\infty}\int_{B_r(x_i)}\abs{f_n}\dv\geq4\pi.
		\end{equation*}
	\end{enumerate}
\end{remark}

\bigskip
	
We now apply Lemma \ref{minimalmass} to $\zeta_n=v_n$ defined in \eqref{vn} with $f_n=\lambda\big(he^{u_n}-1\big)$. In this case:
\begin{enumerate}
	\item either $\int_M e^{q(v_n-\bar{v}_n)}\dv\leq C$;
	\item or else $\exists \,\,  x_1,\dots,x_l\in M$ such that $\liminf_{n\to\infty}\int_{B_r(x_i)}\lambda\abs{he^{u_n}-1}\dv\geq4\pi$ $\forall r>0$.
\end{enumerate}
From now on, assume $(u_n)_n$ to be unbounded in $H^1(M)$ (otherwise there is nothing to prove); we then want to rule out the first alternative above. By our assumption and \eqref{functional}, we easily see that $\bar{u}_n\to-\infty$, so, by virtue of \eqref{vn}, $\bar{v}_n\to-\infty$ as well as $n\to+\infty$, so that the first alternative implies
\begin{equation*}
	\int_M e^{pv_n}\dv=o_n(1) \qquad \forall p\in [1,q].
\end{equation*}
Thus, given any $p\in (1,q)$, we can estimate
\begin{equation}\label{exponentialestimate}
	\begin{split}
		\int_M\abs{e^{pv_n}-e^{pu_n}}\dv&=\int_M\abs{\int_0^1p(u_n-v_n)e^{p(tu_n+(1-t)v_n)}\,dt}\dv \\
		&\leq \int_M p\abs{u_n-v_n}e^{pv_n}e^{p\abs{u_n-v_n}}\dv \\
		&\leq p\normps{u_n-v_n}{a}{M}\left(\int_M e^{qv_n}\dv\right)^{\frac{p}{q}}\left(\int_M e^{bp\abs{u_n-v_n}}\dv\right)^{\frac{1}{b}} \\
		&\leq p\normps{u_n-v_n}{a}{M}\left(\int_M e^{qv_n}\dv\right)^{\frac{p}{q}}\left(Ce^{\frac{b^2p^2}{16\pi}\norm{\nabg\abs{u_n-v_n}}^2_{L^2(M)}}\right)^\frac{1}{b},
	\end{split}
\end{equation}
where we used Lagrange's theorem, H\"{o}lder's inequality with $\frac{p}{q}+\frac{1}{a}+\frac{1}{b}=1$, $a,b>1$, and Moser-Trudinger's inequality \eqref{moser-trudinger-inequality}. In particular, when applying \eqref{moser-trudinger-inequality} we implicitly use the fact that $\int_M\abs{u_n-v_n}\dv\leq\norm{u_n-v_n}_{H^1(M)}=o_n(1)$ (from Lemma \ref{approxlemma}) in order to deal with the mean value.
Thus, using Lemma \ref{approxlemma} together with Sobolev's embedding Theorem, it follows that
\begin{equation*}
	\int_M \abs{e^{pv_n}-e^{pu_n}}\dv=o_n(1).
\end{equation*}
Using this formula, \eqref{normalization} and Jensen's inequality we get
\begin{equation*}
	\frac{1}{\normps{h}{\infty}{M}^p}\leq\left(\int_Me^{u_n}\dv\right)^p\leq\int_M e^{pu_n}\dv=\int_M e^{pv_n}\dv+o_n(1)=o_n(1),
\end{equation*}
which is a contradiction.

In other words, the second alternative of Lemma \ref{minimalmass} must hold whenever we consider an $H^1$-unbounded sequence $(u_n)_n$ as above; in particular, we must have concentration of mass at some point of $M$.

	\section{Isolation of concentration profiles}\label{sectionbubbleextr}
	
From now on, $(u_n)_n$ will be an \emph{unbounded} Palais-Smale sequence of $J_\lambda$ satisfying \eqref{normalization}. We want to isolate a blow-up profile of $(u_n)_n$.
	 Take $\rho\in(0,4\pi)$; since \eqref{minmasseq} holds for $f_n=\lambda(he^{u_n}-1)$ (see the end of previous Section), there exist points $(x_n)_n$ and radii $(r_n)_n$ satisfying
	
	\begin{equation}\label{rhodef}
		\int_{B_{r_n}(x_n)}\lambda(he^{u_n}-1)\dv =\sup_{x\in M} \int_{B_{r_n}(x)}\lambda(he^{u_n}-1)\dv=\rho.
	\end{equation} 
	Clearly $r_n\to0$ as $n\to\infty$. Without loss of generality, we can assume that $x_n\to x\in M$ as $n\to\infty$. We can also assume that there exists $(\tilde{r}_n)_n$ satisfying
	
	\begin{equation}\label{scalingassumptions}
		\frac{r_n}{\tilde{r}_n}\xrightarrow{n\to\infty}0,\qquad \tilde{r}_n\xrightarrow{n\to\infty}0 \quad \text{and}\quad \int_{B_{r_n}(y)}\lambda(he^{u_n}-1)\dv<4\pi \quad \forall y\in B_{\tilde{r}_n}(x_n).
	\end{equation}
	
	We now localise our problem: for any $n\in \N$, pick an isothermal local chart $(U_n,\psi_n)$ centered at $x_n$ and such that $\psi_n(U_n)=B_\delta(0)$ for some (small and uniform) $\delta>0$. Let $\phi_n$ be the conformal factor, that is, $ds_g^2=e^{\phi_n}dx^2$, and define $\eta_n$ to be the solution of
	
	\begin{equation}\label{etadef}
		\begin{cases*}
			\Delta \eta_n= e^{\phi_n} & \text{in $B_\delta$,} \\
			\eta_n=0                         &  \text{in $\partial B_\delta$}.
		\end{cases*}
	\end{equation} 
Define
\begin{equation}\label{localisedfunctions}
	\begin{cases*}
		\tilde{v}_n:=v_n\circ\psi_n^{-1}-\lambda \eta_n, \\
		\tilde{V}_n:=\lambda(h\circ\psi_n^{-1})e^{\phi_n+\lambda \eta_n}, \\
		\tilde{u}_n= u_n\circ\psi_n^{-1}-\lambda \eta_n.
	\end{cases*}
\end{equation}
Then one easily checks that
\begin{equation*}
	\begin{cases*}
		-\Delta \tilde{v}_n=\tilde{V}_ne^{\tilde{u}_n} & \text{in $B_\delta(0)$,} \\
		\int_{\psi_n(B_{r_n}(x_n))}\tilde{V}_n e^{\tilde{u}_n} \,dx=\rho + o_n(1),
	\end{cases*}
\end{equation*}
where the second formula follows from \eqref{rhodef} and a change of variables.

We now proceed to rescale our sequences: let
\begin{equation}\label{localisedscaledfunctions}
	\begin{cases*}
		\hat{v}_n(x):=\tilde{v}_n(r_nx)+2\log r_n & \text{for $x\in B_{\frac{\delta}{r_n}}(0)$}, \\
		\hat{u}_n(x):=\tilde{u}_n(r_nx)+2\log r_n,
	\end{cases*}
\end{equation}
so that
\begin{equation}\label{scaledequation}
	\begin{cases*}
		-\Delta \hat{v}_n=\tilde{V}_n(r_nx)e^{\hat{u}_n(x)}=: \hat{V}_n(x)e^{\hat{u}_n(x)} & \text{in $B_\frac{\delta}{r_n}(0)$,} \\
		\int_{\frac{1}{r_n}\psi_n(B_{r_n}(x_n))}\hat{V}_ne^{\hat{u}_n}\,dx=\rho +o_n(1),
	\end{cases*}
\end{equation}
	where again the second equality follows from a change of variables. We also notice that 
	\begin{equation}\label{starpap}
	\frac{1}{r_n}\psi_n(B_{r_n}(x_n))\longrightarrow B_1(0)\subseteq\R^2 \quad \text{as $n\to\infty$},
	\end{equation}
	 in the sense of $L^1$-convergence of indicator functions.
	
	We can now state the main result of this Section:
	
	\begin{prop}\label{bubbleextractionthm}
		Let $(u_n)_n$, $(v_n)_n$, $(\hat{u}_n)_n$, $(\hat{v}_n)_n$, $(r_n)_n$, $(\tilde{r}_n)_n$ be defined as above. Then there exists a sequence $(b_n)_n$ of real numbers, a positive costant $\mu>0$, $x_0\in \R^2$ and $\alpha\in (0,1)$ such that, up to a subsequece, $\hat{v}_n+b_n\longrightarrow\hat{w}$ in $C^{\alpha}_{loc}(\R^2)$ as $n\to\infty$, where $\hat{w}$ is of the form
		\begin{equation}\label{bubbleequation}
			\hat{w}(x)=\log\frac{\tau}{\left(1+\frac{\tau}{8}\abs{x-x_0}^2\right)^2}-\log \mu,
	\end{equation}	 
	for some $\tau>0$; in particular, $\hat{w}$ is a solution of
	\begin{equation}\label{liouvilleeq}
		\begin{cases*}
			-\Delta w=\mu e^w & \text{in $\R^2$,} \\
			\int_{\R^2} \mu e^w \,dx<+\infty.
		\end{cases*}	
	\end{equation}	
		Moreover, if $R_n\to\infty$ sufficiently slowly, then 
		\begin{equation}\label{bubblemass}
			\int_{B_{R_nr_n}(x_n)}\lambda(he^{u_n}-1)\dv\to 8\pi \quad \text{as $n\to\infty$.}
		\end{equation}
	\end{prop}

	 For the proof we will need the following lemma:
	 
	 \begin{lemma}\label{gradient estimate}
	 	Consider $(u_n)_n$, $(v_n)_n$ as above, and suppose $1\leq p<2$. There exists a constant $C=C(p,M,g)$ such that, for $r>0$ sufficiently small and $\forall q\in M$, there holds
	 	\begin{equation*}
	 		\int_{B_r(q)}\abs{\nabg v_n}^p\dv\leq Cr^{2-p}.
	 	\end{equation*}
	 \end{lemma}
	 
	 \begin{proof}
	 	Write $f_n:=\lambda(he^{u_n}-1)$; then $(f_n)_n$ is uniformly bounded in $L^1(M)$. By Green's representation formula \eqref{greenrepresentationformula} and decay estimates \eqref{greenfunctiondecay}, we get
	 	\begin{equation*}
	 			\abs{\nabg v_n(x)}\leq C\int_M\frac{\abs{f_n(y)}}{d_g(x,y)}\dv(y)=C\int_M \frac{\normps{f_n}{1}{M}}{d_g(x,y)}\frac{\abs{f_n(y)}}{\normps{f_n}{1}{M}}\dv(y), 
	 	\end{equation*}
	 	and, by virtue of Jensen's inequality, 
	 	\begin{equation*}
	 		\begin{split}
	 			\int_{B_s(q)}\abs{\nabg v_n(x)}^p\dv(x)&\leq C\int_{B_s(q)}\left(\int_M \frac{\normps{f_n}{1}{M}^{p-1}\abs{f_n(y)}}{d_g(x,y)^p}\dv(y)\right)\dv(x) \\
	 			&=C\normps{f_n}{1}{M}^{p-1}\int_M \left(\int_{B_s(q)}\frac{1}{d_g(x,y)^p}\dv(x)\right)\abs{f_n(y)}\dv(y) \\
	 			&\leq C\sup_{y\in M}\int_{B_s(q)}\frac{1}{d_g(x,y)^p}\dv(x)\leq C(p,M,g)s^{2-p}.
	 		\end{split}
	 	\end{equation*}
	 Here, being $M$ closed, we can take a constant $C(p,M,g)$ which does not depend upon $q$ whenever the radius $s$ is small enough.
	 \end{proof}
	 
	 \begin{proof}[Proof of Proposition \ref{bubbleextractionthm}]
	 	Let $R>1$ be fixed and, for $n$ large enough, define $b_n$ in such a way that
	 	\begin{equation}\label{w_nhat}
	 		\hat{w}_n:=\hat{v}_n+b_n \quad \text{satisfies} \quad a_n:=\frac{1}{\abs{B_{2R}}}\int_{B_{2R}}\hat{w}_n\,dx= \frac{1}{\abs{B_{2R}}}\int_{B_{2R}}\hat{u}_n\,dx \quad \forall n\in \N.
	 \end{equation}	
	 By construction,
	 \begin{equation*}
	 	-\Delta \hat{w}_n=\hat{V}_ne^{\hat{u}_n} \qquad \text{in $B_{2R}(0)$.}
	 \end{equation*}
	 Moreover, from \eqref{w_nhat} and the scaling invariance of the mass, 
	 \begin{equation*}
	 	\frac{1}{\abs{B_{2R}}}\int_{B_{2R}}\hat{w}_n\leq\frac{1}{\abs{B_{2R}}}\int_{B_{2R}}e^{\hat{u}_n}\leq C \quad \forall n\in \N,
	 \end{equation*}
	 that is, $a_n\leq C$.
	 We also notice that
	 \begin{equation}\label{scaledapproximation}
	 	\norm{\hat{w}_n-\hat{u}_n}_{H^1(B_{2R})}=o_n(1),
	 \end{equation}
	 which is a consequence of Poincaré inequality, a scaling argument and Lemma \ref{approxlemma}. 
	 
	 	At this point, choose a smooth radial cutoff function $\psi_R$ such that $\psi_R\equiv 1$ on $B_R$, $\psi_R\equiv 0$ on $\R^2\backslash B_{2R}$ and $0\leq\psi_R\leq 1$. Define
	 	\begin{equation}\label{xi_ndef}
	 		\xi_n:=\psi_R\hat{w}_n+(1-\psi_R)a_n, \qquad
	 		\hat{\xi}_n:=\xi_n-a_n.
	 	\end{equation}
	 	In particular, we notice that $\hat{\xi}_n$ is compactly supported in $B_{2R}$. One has
	 	\begin{equation*}
	 			\begin{split}
	 		-\Delta \hat{\xi}_n&=-\Delta\psi_R \hat{w}_n+\psi_R(-\Delta\hat{w}_n)-2\nabla\psi_R\nabla\hat{w}_n+a_n\Delta\psi_R \\
	 		&=\psi_R\hat{V}_ne^{\hat{u}_n}-2\nabla\psi_R\nabla\hat{w}_n-\Delta\psi_R\hat{w}_n+a_n\Delta\psi_R \\
	 		&=:\psi_R\hat{V}_ne^{\hat{u}_n}-\hat{f}_n.
	 \end{split}	
	\end{equation*}
	
For any $p\in[1,2)$, using the definition of $\hat{w}_n$ we see that
\begin{equation*}
	\begin{split}
		\int_{B_{2R}}\abs{\hat{f}_n}^p\,dx&\leq C_p\left[\int_{B_{2R}}\abs{\nabla\psi_R\nabla\hat{v}_n}^p\,dx+\int_{B_{2R}}\abs{\Delta\psi_R}^p\abs{\hat{w}_n-a_n}^p\,dx\right] \\
		&\leq C_{p,R}\left[\int_{B_{2R}}\abs{\nabla\hat{v}_n}^p\,dx+\int_{B_{2R}}\abs{\hat{w}_n-a_n}^p\,dx\right] \\
		&\leq \tilde{C}_{p,R}\int_{B_{2R}}\abs{\nabla\hat{v}_n}^p\,dx,
	\end{split}
\end{equation*}
where the last estimate follows from Poincaré inequality. Finally, using Lemma \ref{gradient estimate} with $r=2r_nR$ together with a scaling argument, one has
\begin{equation*}
	\int_{B_{2R}}\abs{\hat{f}_n}^p\,dx\leq C=C(p,R) \qquad \text{uniformly in $n\in \N$}.
\end{equation*}
From this formula and $p>1$, we easily get
\begin{equation}\label{triangolorovesciato}
	\limsup_{r\to0^+}\,\, \limsup_{n\to\infty}\int_{B_r(y)}\abs{\hat{f}_n}\,dx=0 \qquad \forall y\in B_{2R}, \,\, \forall B_r(y)\subseteq B_{2R}.
\end{equation}
At this point, by \eqref{scaledequation} and \eqref{triangolorovesciato},
\begin{equation*}
	\limsup_{r\to0^+} \,\, \liminf_{n\to\infty}\int_{B_r(y)}\abs{\psi_R\hat{V}_ne^{\hat{u}_n}-\hat{f}_n}\,dx<4\pi,
\end{equation*}
allowing the use of Remark \ref{localminmasslemma} to conclude that there exists $q>1$ (close to $1$) and a positive constant $C>1$ such that 
\begin{equation*}
	\int_{B_{R}}e^{q\hat{\xi}_n}\,dx\leq C \qquad \text{uniformly in $n\in \N$,}
\end{equation*}
	which, together with $a_n\leq C$ and \eqref{xi_ndef}, implies
	\begin{equation}\label{quadr}
		\int_{B_R}e^{q\hat{w}_n}\,dx\leq C \quad \forall n\in \N.
	\end{equation}
	Using \eqref{quadr}, we can argue as in \eqref{exponentialestimate} to show that, $\forall \,\, p\in[1,q)$, one has
	\begin{equation}\label{quadrato}
		\int_{B_R} \abs{e^{p\hat{w}_n}-e^{p\hat{u}_n}}\,dx=o_n(1).
	\end{equation}

We now estimate $a_n$ from below: indeed, using \eqref{xi_ndef}, \eqref{quadrato} and \eqref{scaledequation}, we get

\begin{equation*}
		Ce^{a_n}\geq e^{a_n}\int_{B_{R}}e^{\hat{\xi}_n}=\int_{B_{R}}e^{\xi_n}=\int_{B_{R}}e^{\hat{w}_n}=\int_{B_{R}}e^{\hat{u}_n}+o_n(1)\geq\frac{\rho}{2}. 
\end{equation*}
 Hence
\begin{equation*}
	e^{a_n}\geq\frac{\rho}{2C} \Longrightarrow a_n\geq\log\left(\frac{\rho}{2C}\right),
\end{equation*}
which, together with the estimate from above, implies that there exists $C>0$ such that
\begin{equation}\label{circ}
	\abs{a_n}\leq C \qquad \text{uniformly in $n\in \N$.}
\end{equation}

Finally, using elliptic estimates, \eqref{quadr}, \eqref{quadrato}, \eqref{circ}, Poincaré inequality and Lemma \ref{gradient estimate}, we conclude that
\begin{equation*}
		\norm{\hat{w}_n}_{W^{2,p}(B_{\frac{R}{2}})}\leq C\left(\norm{\hat{w}_n}_{L^p(B_R)}+\norm{\hat{V}_ne^{\hat{u}_n}}_{L^p(B_R)}\right)\leq C.
\end{equation*}
Thus $\hat{w}_n$ is uniformly bounded in $L^p(B_R)$ for any $R>0$ fixed, allowing us to use once more elliptic estimates and a diagonal argument to show that there exists $\hat{w}\in C^{\alpha}_{loc}(\R^2)\cap H^1_{loc}(\R^2)$ and a subsequence $\hat{w}_n$ (not relabeled) such that
\begin{equation*}
	\hat{w}_n\xrightarrow{n\to\infty}\hat{w} \quad \text{in $C^{\alpha}_{loc}(\R^2)\cap H^1_{loc}(\R^2)$.}
\end{equation*}

In order to conclude, it remains to show that $\hat{w}$ solves \eqref{liouvilleeq}. Multiply $-\Delta \hat{w}_n=\hat{V}_ne^{\hat{u}_n}$ by a smooth test function $\varphi\in C_c^{\infty}(\R^2)$ with $\supp(\varphi)\ssubset B_R(0)$ and integrate by parts to get
\begin{equation}\label{weakformulation}
	\int_{B_R}\nabla\hat{w}_n\nabla\varphi\,dy=\int_{B_R}\hat{V}_ne^{\hat{u}_n}\varphi\,dy=\int_{B_R}\hat{V}_ne^{\hat{w}_n}\varphi\,dy+o_n(1),
\end{equation}
where we used \eqref{quadrato} in the last passage.
Here $n$ is supposed to be large enough so that the integrals are well-defined. By definition $\hat{V}_n(y)=\lambda(h\circ\psi_n^{-1})(r_ny)e^{(\phi_n+\lambda\eta_n)(r_ny)}$, and, since $\psi^{-1}(0)=x_n$ and $x_n\to x\in M$, we see that $\hat{V}_n$ is uniformly bounded in $C^0_{loc}(\R^2)$ and it converges to some positive costant $\mu>0$. More precisely, if $\phi_n\to\bar{\phi}$ in $C^0(B_{\frac{\delta}{2}}(0))$ and $\eta_n\to\bar{\eta}$ in $C^0(B_{\frac{\delta}{2}}(0))$, then $\mu=\lambda h(x)e^{(\bar{\phi}+\lambda\bar{\eta})(0)}>0$. Hence we can pass to the limit in \eqref{weakformulation} and get
\begin{equation*}
	\int_{B_R}\nabla\hat{w}\nabla\varphi\,dy=\int_{B_R}\mu e^{\hat{w}}\varphi\,dy.
\end{equation*} 
Being $\varphi\in C_c^{\infty}(\R^2)$ arbitrary, we see that $\hat{w}$ is a distributional solution to $-\Delta\hat{w}=\mu e^{\hat{w}}$ in $\R^2$. Moreover, Fatou's lemma and a change of variables give
\begin{equation*}
	\int_{B_R} \mu e^{\hat{w}}\,dx\leq\liminf_n\int_{B_{R}}\hat{V}_ne^{\hat{w}_n}\leq\liminf_n\int_{B_{\frac{\delta}{r_n}}}\hat{V}_ne^{\hat{u}_n}\leq \lambda \quad \forall R>0.
\end{equation*}
Hence $\hat{w}$ is a distributional (a posteriori classical) solution of \eqref{liouvilleeq}. At this point, we can use the classification result of \cite{chenli-Duke-1991-classification} to infer that $\hat{w}$ has the form \eqref{bubbleequation} and, moreover, $\int_{\R^2}\mu e^{\hat{w}}=8\pi$.

Finally, if $R_n\to+\infty$ slowly enough, then
\begin{equation*}
	\begin{split}
	\int_{B_{R_nr_n}(x_n)}\lambda(he^{u_n}-1)\dv&=\int_{\psi_n(B_{R_nr_n}(x_n))}\tilde{V}_ne^{\tilde{u}_n}\,dx -\lambda Vol_g(B_{R_nr_n}(x_n)) \\
	&=\int_{\frac{1}{r_n}\psi_n(B_{R_nr_n}(x_n))}\hat{V}_ne^{\hat{u}_n}\,dx -\lambda Vol_g(B_{R_nr_n}(x_n)) \\
	&\to\int_{\R^2}\mu e^{\hat{w}}\,dx=8\pi \quad \text{as $n\to\infty$},
\end{split}
\end{equation*}
which shows \eqref{bubblemass} and concludes our proof.
	 \end{proof}

\begin{remark}
		 If $(u_n)_n$ is as in Proposition \ref{bubbleextractionthm}, it is actually possible to slightly modify the sequence of points $(x_n)_n$ and radii $(r_n)_n$ in order to have $x_0=0$ and $\tau=1$ in \eqref{bubbleequation}. We will always assume this normalization condition to hold in the next Sections.
\end{remark}

\section{Quantization in a simple blow-up}\label{sectionsimpleblowup}	 
	 In the previous Section we showed that, given an unbounded (PS)-sequence $(u_n)_n$, it is possible to extract a bubble of ``mass'' $8\pi$ around each concentration point. Following \cite{malchiodi-2006-CRELLE}, we now introduce an integral notion of ``simple" blow-up point (the original definition, due to Schoen, is slightly different).
	 
	 \begin{defin}\label{simpleblowupdefin}
	 	Let $(u_n)_n$, $(\hat{w}_n)_n$ be as in Section \ref{sectionbubbleextr}. Given three sequences $(x_n)_n\subseteq M$ of points and $(r_n)_n$, $(s_n)_n$ of positive radii with $r_n,s_n\leq i(M)$ (injectivity radius of $M$), we say that they are a \emph{simple blow-up} for $(u_n)_n$ if:
	 
	 	\begin{gather}
	 		r_n\to 0, \quad \frac{r_n}{s_n}\to 0 \quad \text{as $n\to\infty$}; \label{simpleblowupradii}\\
	 		\exists R_n\to\infty \quad \text{such that} \,\, \frac{R_nr_n}{s_n}\to0 \,\,\, \text{and} \quad \norm{\hat{w}_n-\log\frac{\mu}{\left(1+\frac{1}{8}\abs{\cdot}^2\right)^2}}_{C^\alpha(B_{R_n})} \xrightarrow{n\to\infty}0; \label{simpleblowupbubble}\\
	 		\forall\rho>0 \,\, \exists C_\rho>0 \quad \text{such that}\quad \int_{B_s(y)}he^{u_n}\dv\geq\rho \quad \text{and} \quad B_s(y)\subseteq B_{s_n}(x_n)\backslash B_{R_nr_n}(x_n) \notag \\
	 		\text{imply} \quad s\geq \frac{1}{C_\rho}d_g(y,x_n). \label{simpleblowupmass2}
	 	\end{gather}
	 \end{defin}
	 
	 The following Proposition, which is the analogue for (PS)-sequences of Proposition 4.2 of \cite{malchiodi-2006-CRELLE} and of Proposition 2 in \cite{lishafrir-1994}, asserts that, in a simple blow-up, there is no further concentration of mass outside the ``central" bubble:
	 
	 \begin{prop}\label{simpleblowuptheorem}
	 	Let $(u_n)_n$ be a (PS)-sequence for $J_\lambda$ satisfying \eqref{normalization} and \eqref{morseindexproperty}, and let $(x_n)_n$, $(r_n)_n$, $(s_n)_n$ be a simple blow-up for $(u_n)_n$. There exists a positive constant $C>1$ such that, up to a subsequence,
	 	\begin{equation}\label{simpleblowupquantization}
	 		\int_{B_\frac{s_n(x_n)}{C}} \lambda he^{u_n}\dv=8\pi+o_n(1).
	 	\end{equation}
	 		 \end{prop}
	 	
	 	\begin{proof}
	 		By virtue of Proposition \ref{bubbleextractionthm} and \eqref{simpleblowupbubble}, we already know that, up to a subsequence,
	 		\begin{equation*}
	 			\int_{B_{R_nr_n}(x_n)}\lambda he^{u_n}\dv=8\pi+o_n(1);
	 		\end{equation*}	
	 		 then, since $(x_n)_n$, $(r_n)_n$, $\frac{1}{C}(s_n)_n$ is trivially a simple blow-up $\forall C>1$, we only need to show that there is no accumulation of mass in the \emph{neck region} $B_{\frac{s_n}{C}}(x_n)\backslash B_{R_nr_n}(x_n)$ for $C$ large enough.
	 		
	 		 Assume by contradiction that $\forall C>1$ there exists $\delta_C>0$ satisfying
	 		\begin{equation}\label{contradictionassumption}
	 			\liminf_{n\to\infty}\int_{B_{\frac{s_n}{C}}(x_n)\backslash B_{R_nr_n}(x_n)}he^{u_n}\dv\geq\delta_C.
	 		\end{equation}
	 	By definition of simple blow-up (see \eqref{simpleblowupmass2}), there exist $(\tau_{n}^C)_n$ and $L>1$ such that, given
	 	\begin{equation*}
	 		A_{n,C}:=\left\{x\in M\mid \frac{\tau_n^C}{L}\leq d_g(x,x_n)\leq L\tau_n^C\right\},
	 	\end{equation*}
	 	one has
	 	\begin{equation*}
	 		\int_{A_{n,C}} he^{u_n}\dv=\max_{LR_nr_n\leq \tau\leq\frac{s_n}{CL}}\int_{\left\{\frac{\tau}{L}\leq d_g(x,x_n)\leq L\tau\right\}} he^{u_n}\dv=:\sigma_{n,C}<\sigma_C,
	 	\end{equation*}
	 	$\forall n\in \N$ and for a small $\sigma_C\in(0,\delta_C)$. 
	Up to a subsequence, we then face two possibilities:	 	
	
	 	$\mathit{(i)}$ Either
	 	\begin{equation}\label{case1mass}
	 		\lim_{n\to\infty}\int_{A_{n,C}} he^{u_n}\dv\geq\delta_C' \quad \forall n\in \N, \quad \text{for some $\delta_C'\in (0,\sigma_C]$};
	 	\end{equation}
	 	
	 	$\mathit{(ii)}$ or else 
	 	\begin{equation}\label{case2assumption}
	 		\lim_{n\to\infty}\int_{A_{n,C}}he^{u_n}\dv=0.
	 	\end{equation}	  
	 	Heuristically, in case $\mathit{(i)}$ there is a residual mass concentrating in a sequence of dyadic rings of fixed relative width, while in case $\mathit{(ii)}$ the residual mass spreads and vanishes in any sequence of dyadic rings contained in the neck.	  
	 	We now proceed to rule out both cases.
	 	
	 	\medskip
	 	
	 	\underline{Case $(i)$.} Notice that, if $R_n\to+\infty$ slowly enough, then, by \eqref{simpleblowupbubble},
	 		\begin{equation}\label{reversetrianglepap}
	 		\frac{\tau_n^C}{R_nr_n}\to+\infty \quad \text{as} \quad n\to\infty.
	 	\end{equation}	
	 	 Arguing as in Section \ref{sectionbubbleextr}, we localise our problem around $x_n$, defining $\tilde{v}_n$, $\tilde{u}_n$ and $\tilde{V}_n$ as in \eqref{localisedfunctions}. Then, if $\sigma_C$ small enough, \eqref{case1mass}, \eqref{simpleblowupmass2} and \eqref{starpap} imply
	 	\begin{equation}\label{case1localmass}
	 		4\pi>\liminf_{n\to\infty}\int_{\left\{\frac{\tau_n^C}{2L}\leq\abs{x}\leq 2L\tau_n^C\right\}} \tilde{V}_n e^{\tilde{u}_n}\,dx\geq\delta_C'.
	 	\end{equation}
	 	The idea now is to exploit \eqref{case1localmass} and a scaling argument in order to obtain convergence of a scaled sequence towards the solution of a singular limit problem which will provide the desired contradiction. 
	 	Let $C_m\to+\infty$ as $m\to\infty$ and define $\tau_{n,m}:=\tau_n^{C_m}$.
	 	Let 
	 	\begin{equation}\label{hat1,nfunctions}
	 		\begin{cases*}
	 	\hat{v}_{n,m}(x):=\tilde{v}_{n}(\tau_{n,m}x)+2\log(\tau_{n,m}), \\
	 	 \hat{u}_{n,m}(x):=\tilde{u}_{n}(\tau_{n,m}x)+2\log(\tau_{n,m});
	 	\end{cases*}
	 	\end{equation}
	 	 then
	 	\begin{equation*}
	 		\begin{cases*}
	 			-\Delta \hat{v}_{n,m}=\tilde{V}_{n}(\tau_{n,m}x)e^{\hat{u}_{n,m}} &\text{in $B_{\frac{\delta}{\tau_{n,m}}(0)}$,} \\
	 			\int_{B_{\frac{\delta}{\tau_{n,m}}(0)}}\tilde{V}_{n}(\tau_{n,m}x)e^{\hat{u}_{n,m}}\,dx\leq C<+\infty.
	 		\end{cases*}
	 	\end{equation*}	 	
	 	
	 	\medskip
	 	
	 	\begin{lemma}\label{singularbubbleextraction}
	 		There exist a subsequence $(n(m))_m$ and constants $b_{n,m}$ such that $\hat{v}_{n(m),m}+b_{n(m),m}\xrightarrow{m\to\infty}\hat{w}$ in $C^{\alpha}_{loc}(\R^2\backslash\{0\})$, where $\hat{w}$ satisfies
	 		\begin{equation}\label{singularlimitequation}
	 			\begin{cases*}
	 				-\Delta \hat{w}=\mu e^{\hat{w}}+\beta \delta_0 &\text{in $\R^2$,} \\
	 				\int_{\R^2}\mu e^{\hat{w}}\,dx<+\infty,
	 			\end{cases*}
	 		\end{equation}
	 		for some $\mu>0$ and $\beta\in [8\pi,\lambda)$.
	 	\end{lemma}
	 	The proof of Lemma \ref{singularbubbleextraction} is similar to that of Proposition \ref{bubbleextractionthm}, but this time we need to remove a tiny neighborhood of the origin because the scaled sequence is still unbounded around $0$. To keep the thread going, we postpone it to the Appendix.
	 	
	 	\medskip
	 	
	 	Given \eqref{singularlimitequation}, we can use Green's representation formula for $\hat{w}$ in the unitary ball $B_1(0)$: the Green's function of Dirichlet laplacian on the unitary ball is given by
	 	\[
	 	G(x,y)=\frac{1}{2\pi}\left(\log\frac{1}{\abs{x-y}}+\log\abs{\abs{x}y-\frac{x}{\abs{x}}}\right) \quad \text{for $x\not=y\in B_1(0)$ $x\not=0$},
	 	\]
	 	see \cite{evansbook}, Chapter 2. Therefore
	 	\begin{equation*}
	 		\begin{split}
	 		\hat{w}(x)&=\int_{B_1(0)}(\mu e^{\hat{w}(y)}+\beta\delta_0(y))G(x,y)\,dy-\int_{\partial B_1(0)}\hat{w}(y)\frac{\partial G}{\partial\nu}(x,y)\, d\sigma \\
	 		&=\frac{\beta}{2\pi}\log\frac{1}{\abs{x}}+\int_{B_1(0)}\frac{\mu}{2\pi}e^{\hat{w}(y)}G(x,y)\,dy +\obig(1).
	 		\end{split}
	 	\end{equation*}
	 	As a consequence, being $G(x,y)>0$ for $x,y\in B_1(0)$, $\exists \, C>0$ such that
	 	\[
	 	\hat{w}(x)\geq\frac{\beta}{2\pi}\log\frac{1}{\abs{x}}-C \qquad \text{in $B_1(0)\backslash\{0\}$}.
	 	\]
	 	Thus $e^{\hat{w}(x)}\geq\frac{e^{-C}}{\abs{x}^{\frac{\beta}{2\pi}}}$ in $B_1(0)\backslash\{0\}$, contradicting the integrability condition in \eqref{singularlimitequation}.
	 	
	 		\bigskip
	 		
	 		\underline{Case $(ii)$.} In this case the idea is to reach a contradiction to the Morse index condition \eqref{morseindexproperty}, which we recall to be
	 		\begin{equation*}
	 			\sup\left\{\dim(E)\mid E \text{ subspace of } H^1(M), J_\lambda^{''}(u_n)[w,w]<-\frac{1}{n}\norm{w}_{H^1}^2 \,\, \forall w\in E \right\}\leq N \,\, \forall n\in\N.
	 		\end{equation*}
	 		As a consequence, we see that, $\forall n\in \N$, $J_\lambda^{''}(u_n)$ has at most $N$ eigenvalues below $-\frac{1}{n}$.
	 	 Fix any constant $C>1$; as usual, we can localise our problem around $x_n$ and, as a consequence of \eqref{case2assumption}, for any fixed $\delta'\in (0,\delta_C)$ and $R_nr_n<a_n<b_n<s_n$ satisfying 
	 		\begin{equation}\label{case2localisedassumption}
	 		\liminf_{n\to\infty}\int_{a_n\leq\abs{x}\leq b_n} \tilde{V}_n e^{\tilde{u}_n}\,dx\geq\delta',
	 	\end{equation}	
	 	then $\frac{a_n}{b_n}\to 0$ necessarily as $n\to\infty$.
	 	
	 	Let $\gamma:=\frac{\delta'}{10N}$, where $\delta'$ is given by \eqref{case2localisedassumption}. Using the continuity of $r\to\int_{B_r(0)}\tilde{V}_ne^{\tilde{u}_n}$, we can find sequences $(a_{l,n}^i)_n$, $l=1,\dots,N+1$, $i=1,\dots,7$, satisfying the following properties:
	 	\begin{gather}
	 		a_n=a_{1,n}^1 \,\, \forall n, \quad 1>\frac{a_{l,n}^{i}}{a_{l,n}^{i+1}}\xrightarrow{n\to\infty}0 \quad \forall i=1,\dots,6,  \,\,\forall \, n\in \N; \label{extremalschoice1} \\
	 		a_{l,n}^7\leq a_{l+1,n}^1 \quad \forall \, l=1,\dots,N, \,\, \forall n\in\N; \label{extremalschoice2} \\
	 		\int_{a_{l,n}^2\leq\abs{x}\leq a_{l,n}^3}\tilde{V}_ne^{\tilde{u}_n}=\int_{a_{l,n}^5\leq\abs{x}\leq a_{l,n}^6}\tilde{V}_ne^{\tilde{u}_n}=\gamma; \label{masschoicebig} \\
	 		\int_{a_{l,n}^i\leq\abs{x}\leq a_{l,n}^{i+1}}\tilde{V}_ne^{\tilde{u}_n}=\frac{\gamma}{100} \quad \forall i\in\{1,3,4,6\}. \label{masschoicelittle}
	 	\end{gather}
	 		For any $l=1,\dots,N+1$ and $\forall n\in \N$, let 
	 		\begin{equation}\label{testfunctiondef}
	 				\tilde{w}_{l,n}(r):=\begin{cases*}
	 					0 & \text{if $r<a_{l,n}^1$}, \\
	 					-\frac{\log r-\log a_{l,n}^1}{\log a_{l,n}^2-\log a_{l,n}^1} & \text{if $a_{l,n}^1\leq r<a_{l,n}^2$}, \\
	 					-1 & \text{if $a_{l,n}^2\leq r<a_{l,n}^3$}, \\
	 					\frac{\log r-\log a_{l,n}^4}{\log a_{l,n}^4-\log a_{l,n}^3} & \text{if $a_{l,n}^3\leq r < a_{l,n}^4$}, \\
	 					\frac{\log r-\log a_{l,n}^4}{\log a_{l,n}^5-\log a_{l,n}^4} & \text{if $a_{l,n}^4\leq r <a_{l,n}^5$}, \\
	 					1 & \text{if $a_{l,n}^5\leq r< a_{l,n}^6$}, \\
	 					-\frac{\log r -\log a_{l,n}^7}{\log a_{l,n}^7-a_{l,n}^6} & \text{if $a_{l,n}^6\leq r<a_{l,n}^7$}, \\
	 					0 & \text{if $r\geq a_{l,n}^7$},
	 			\end{cases*}
	 		\end{equation}
	 		\begin{figure}
	 			\centering
	 			\begin{tikzpicture}
	 				\draw[gray, thin] (-1,0) -- (14,0);
	 				\draw[black, thin] (-1,0) -- (0,0);
	 				\draw[black, thin] (13,0) -- (14,0);
	 				\draw[black, thin] (0,0) -- (1,-1);
	 				\draw[black, thin] (1,-1) -- (4,-1);
	 				\draw[black, thin] (4,-1) -- (5,0);
	 				\draw[black, thin] (5,0) -- (6.5,1);
	 				\draw[black, thin] (6.5,1) -- (11.5,1);
	 				\draw[black, thin] (11.5,1) -- (13,0);
	 				\draw[->, gray, thin,] (14,0) -- (-1,0);
	 				\filldraw[black] (1,-1) node[anchor=north east]{$-1$};
	 				\filldraw[black] (11.5,1) node[anchor=south west]{$1$};
	 				\filldraw[black] (-1,0) node[anchor=north]{$s=\log r$};
	 				\filldraw[black] (0,0) circle (1pt) node[anchor=south]{$b^1_{l,n}$};
	 				\filldraw[black] (1,0) circle (1pt) node[anchor=south]{$b^2_{l,n}$};
	 				\filldraw[black] (4,0) circle (1pt) node[anchor=south]{$b^3_{l,n}$};
	 				\filldraw[black] (5,0) circle (1pt) node[anchor=north west]{$b^4_{l,n}$};
	 				\filldraw[black] (6.5,0) circle (1pt) node[anchor=north west]{$b^5_{l,n}$};
	 				\filldraw[black] (11.5,0) circle (1pt) node[anchor=north]{$b^6_{l,n}$};
	 				\filldraw[black] (13,0) circle (1pt) node[anchor=north]{$b^7_{l,n}$};
	 			\end{tikzpicture}
	 			\caption{The graph of $\tilde{w}_{l,n}$ in log-polar coordinates $s=\log r$. Here $b^i_{l,n}:=\log a^i_{l,n}$.} \label{figure}
	 		\end{figure}
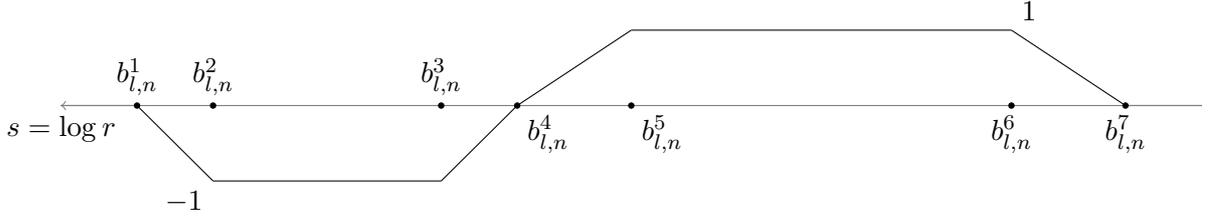
	 		see Figure \ref{figure}. Let then $w_{l,n}:M\to\R$ given by $w_{l,n}(x):=\tilde{w}_{l,n}(\abs{\psi_n(x)})$ for $x\in U_n$ and $w_{l,n}\equiv0$ elsewhere; we recall that $(U_n,\psi_n)$ are the local isothermal charts defined in Section \ref{sectionbubbleextr}.
	 		We can now test the second variation of $J_\lambda$ on those functions. One has
	 		\begin{equation*}
	 			J_\lambda^{''}(u_n)[w,w]=\int_M\abs{\nabg w}^2\dv-\lambda\int_M he^{u_n}w^2\dv+\lambda\left(\int_M he^{u_n}w\dv\right)^2 \quad \forall w\in H^1(M).
	 		\end{equation*}
	 		Hence, for any fixed $l\in\{1,\dots,N+1\}$,
	 		\begin{multline*}
	 				\int_M\abs{\nabg w_{l,n}}^2\dv=\int_{a_{l,n}^1\leq\abs{x}\leq a_{l,n}^7}\abs{\nabla \tilde{w}_{l,n}}^2 e^{-\phi_n}\,dx\leq C\int_{a_{l,n}^1}^{a_{l,n}^7} (\partial_r\tilde{w}_{l,n})^2r\,dr \\
	 				=C\left[\int_{a_{l,n}^1}^{a_{l,n}^2}\frac{dr}{r\left(\log\frac{a_{l,n}^2}{a_{l,n}^1}\right)^2}+\int_{a_{l,n}^3}^{a_{l,n}^4}\frac{dr}{r\left(\log\frac{a_{l,n}^4}{a_{l,n}^3}\right)^2}+\int_{a_{l,n}^4}^{a_{l,n}^5}\frac{dr}{r\left(\log\frac{a_{l,n}^5}{a_{l,n}^4}\right)^2} 
	 				 +\int_{a_{l,n}^6}^{a_{l,n}^7}\frac{dr}{r\left(\log\frac{a_{l,n}^7}{a_{l,n}^6}\right)^2}\right] \\
	 				 =C\left[\frac{1}{\log\frac{a_{l,n}^2}{a_{l,n}^1}}+\frac{1}{\log\frac{a_{l,n}^4}{a_{l,n}^3}}+\frac{1}{\log\frac{a_{l,n}^5}{a_{l,n}^4}}+\frac{1}{\log\frac{a_{l,n}^7}{a_{l,n}^6}}\right]\xrightarrow{n\to+\infty}0 \quad \text{by \eqref{extremalschoice1}.}
	 		\end{multline*}
	 		In other words,
	 		\begin{equation}\label{piece1}
	 			\int_M\abs{\nabg w_{l,n}}^2\dv=o_n(1).
	 		\end{equation}
	 		For the middle term of $J_\lambda^{''}$, we can use \eqref{masschoicebig}, \eqref{masschoicelittle} and \eqref{testfunctiondef} to deduce that
	 		\begin{equation}\label{piece2}
	 			\int_M he^{u_n}w_{l,n}^2\dv\geq2\gamma.
	 		\end{equation}
	 		Similarly, another application of \eqref{masschoicebig}, \eqref{masschoicelittle} and \eqref{testfunctiondef} to the last term of $J_\lambda^{''}$ gives
	 		\begin{equation}\label{piece3}
	 			\left(\int_M he^{u_n}w_{l,n}\dv\right)^2\leq \left(\frac{4}{100}\gamma\right)^2 +o_n(1).
	 		\end{equation}
	 		Finally, from \eqref{piece1}, \eqref{piece2} and \eqref{piece3} we can estimate
	 		\begin{equation*}
	 			J_{\lambda}^{''}(u_n)[w_{l,n},w_{l,n}]\leq -2\lambda\gamma+\lambda\left(\frac{4}{100}\gamma\right)^2+o_n(1),
	 		\end{equation*}
	 		and, being $\norm{w_{l,n}}_{H^1}^2=o_n(1)$ by virtue of \eqref{piece1} and \eqref{testfunctiondef}, we conclude that
	 		\begin{equation*}
	 			J_{\lambda}^{''}(u_n)[w_{l,n},w_{l,n}]<-\lambda\gamma<-\frac{1}{n}\norm{w_{l,n}}_{H^1}^2 \quad \text{for $n$ large enough, $\forall l=1,\dots,N+1$.}
	 		\end{equation*}
	 		Being $(w_{l,n})_n$, $l=1,\dots,N+1$ mutually orthogonal for each $n$ (they have disjoint support by \eqref{extremalschoice2}), we thus obtain a contradiction with \eqref{morseindexproperty}; this completes the proof.
	 	\end{proof}

	 \section{Global quantization and proof of Theorem \ref{maintheorem}}\label{sectionglobalquantization}
	 
	We now proceed to prove Theorem \ref{maintheorem}. Let $\lambda\in(8\pi k,8\pi(k+1))$ and assume by contradiction that $(u_n)_n$ is an \emph{unbounded} Palais-Smale sequence, so that \eqref{minmasseq} holds for $f_n=\lambda(he^{u_n}-1)$. As in \cite{malchiodi-2006-CRELLE}, we proceed in several steps:
	
	\begin{step1*}
		There exists an integer $j\leq k$,  $j$ sequences of points $(x_{1,n})_n,\dots,(x_{j,n})_n\subseteq M$, $x_{i,n}\xrightarrow{n\to\infty}x_i$, radii $(r_{1,n})_n,\dots,(r_{j,n})_n,(\tilde{r}_{1,n})_n,\dots,(\tilde{r}_{j,n})_n$ and positive numbers $\mu_1,\dots,\mu_j>0$ such that, up to a subsequence:
		\begin{gather}
			\frac{\tilde{r}_{i,n}}{r_{i,n}}\to+\infty, \tilde{r}_{i,n}\to0 \,\, \text{as} \,\, n\to\infty \,\,\forall i=1,\dots,j,  \quad B_{\tilde{r}_{i_1,n}}\cap B_{\tilde{r}_{i_2,n}}=\emptyset \,\,\text{for } i_1\not=i_2;  \label{P1} \\
	        \forall R>0, \,\, \hat{w}_{i,n}\to\log\frac{1}{\left(1+\frac{1}{8}\abs{x}^2\right)^2}-\log \mu_i \quad \text{in} \quad H^1(B_R)\cap C^{0,\alpha}(B_R) \quad \text{as $n\to\infty$}; \label{P2}  \\
	        \forall \rho>0 \,\, \exists C_\rho>0 \text{ such that, if } \int_{B_s(y)} he^{u_n}\dv\geq\rho \quad \text{with } B_s(y)\subseteq M\backslash \bigcup_{i=1}^j B_{\tilde{r}_{i,n}}(x_{i,n}),  \notag\\
	        \text{then  } s\geq\frac{1}{C_\rho}d_n(y), \quad \text{where  } d_n(y):=\min_{i=1,\dots,j} d_g(y,{x_{i,n}}). \label{P3} 
	    \end{gather}
	    Here the $\hat{w}_{i,n}$ are defined as in \eqref{w_nhat}.
	\end{step1*}
	 
	 \begin{proof}
	 	Fix $\rho\in (0,4\pi)$ and define $(x_{1,n})_n\subseteq M$ and $(r_{1,n})_n\subseteq \R_{>0}$ satisfying 
	 	\begin{equation*}
	 		\int_{B_{r_{1,n}}(x_{1,n})} he^{u_n}\dv=\max_{x\in M}\int_{B_{r_{1,n}}(x)} he^{u_n}\dv=\rho.
	 	\end{equation*}
	 	Notice that, by virtue of \eqref{minmasseq}, we must have $r_{1,n}\xrightarrow{n\to\infty}0$ necessarily. \\
	    Now, if $\frac{\tilde{r}_{1,n}}{r_{1,n}}\to+\infty$ sufficiently slowly, then the assumptions of Proposition \ref{bubbleextractionthm} are satisfied and we can use it to isolate a blow-up profile. In this way, we obtain \eqref{P1} and \eqref{P2} for $i=1$.
	    At this point, if \eqref{P3} holds for $j=1$, then we are done. 
	    
	    If instead \eqref{P3} does not hold, then $\exists \rho_1\in (0,4\pi)$ and $\exists (y_n)_n\subseteq M$ and $(\bar{r}_n)_n\subseteq \R_{>0}$ such that:
	    \begin{equation*}
	    	\int_{B_{\bar{r}_n}(y_n)} he^{u_n}\dv\geq\rho_1, \qquad B_{\bar{r}_n}(y_n)\subseteq M\backslash B_{\tilde{r}_{1,n}}(x_{1,n}) \quad \text{and} \quad \frac{\bar{r}_n}{d_g(y_n, x_{1,n})}\xrightarrow{n\to\infty}0.
	    \end{equation*}
	    In this case, we can define new sequences of points $(x_{2,n})_n$ and radii $(r_{2,n})_n$ satisfying
	    \begin{equation*}
	    	\int_{B_{r_{2,n}}(x_{2,n})} he^{u_n}\dv=\max_{B_{r_{2,n}}(y)\subseteq M\backslash B_{\tilde{r}_{1,n}}(x_{1,n})}\int_{B_{r_{2,n}}(y)} he^{u_n}\dv=\rho_1.
	    \end{equation*}
	    It is now possible to apply Proposition \ref{bubbleextractionthm} a second time to isolate another concentration profile. Moreover, if $\frac{\tilde{r}_{1,n}}{r_{1,n}}\to+\infty$ slowly enough, one further has
	    \begin{equation*}
	    	\frac{\tilde{r}_{1,n}}{d_g(x_{1,n},x_{2,n})}\longrightarrow0, \qquad \frac{r_{2,n}}{d_g(x_{1,n},x_{2,n})}\longrightarrow0 \quad \text{as $n\to+\infty$};
	    \end{equation*}
	    in particular $r_{2,n}\xrightarrow{n\to\infty}0$. Finally, we can find $\tilde{r}_{2,n}\to0$ such that 
	    \begin{equation*}
	    	\int_{B_{r_{2,n}}(y)} he^{u_n}\dv\leq \rho_1 \quad \forall y\in B_{\tilde{r}_{2,n}}(x_{2,n}), \quad  \frac{\tilde{r}_{2,n}}{r_{2,n}}\to+\infty \quad \text{and} \quad \frac{\tilde{r}_{2,n}}{d_g(x_{1,n},x_{2,n})}\to 0 \quad \text{as $n\to\infty$.}
	    \end{equation*}
	    In other words, the second bubble is separated from the first one, that is, $B_{\tilde{r}_{1,n}}(x_{1,n})\cap B_{\tilde{r}_{2,n}}(x_{2,n})=\emptyset$ $\forall n$ (up to subsequences).
	    
	    We can continue in this way and observe that $j\leq k$ necessarily as each blow-up profile contributes with $8\pi$ of mass and $\lambda<8\pi(k+1)$. This concludes the proof.
	 \end{proof}

	 \begin{step2*}
	 	If $j=1$ in Step 1, then 
	 	\begin{equation*}
	 		\int_M \lambda he^{u_n}\dv=8\pi +o_n(1).
	 	\end{equation*} 
	 \end{step2*}
	 
	 \begin{proof}
	 	Choose $s_{1,n}:=\frac{1}{2}i(M)$ $\forall n\in \N$, where $i(M)$ is the injectivity radius of $M$. Then, by \eqref{P3}, we know that $(x_{1,n})_n$, $(r_{1,n})_n$ and $(s_{1,n})_n$ are a \emph{simple blow-up} for $(u_n)_n$. Therefore we can apply Proposition \ref{simpleblowuptheorem} and obtain that 
	 	\begin{equation}\label{firsteq}
	 		\int_{B_{\frac{1}{2C}i(M)}(x_{1,n})} \lambda he^{u_n}\dv =8\pi +o_n(1),
	 	\end{equation}
	 	for some positive constant $C>1$. By virtue of \eqref{P3}, we can cover $M\backslash B_{\frac{1}{2C}i(M)}(x_{1,n})$ with a finite number $A$ of balls $B_{r_i}(y_i)$, $i=1,\dots,A$ wich is uniform in $n\in \N$ and such that
	 	\begin{equation}\label{triangletrianglepap}
	 		\int_{B_{2r_i}(y_i)}\lambda(he^{u_n}-1)\dv<\pi \qquad \forall i=1,\dots,A.
	 	\end{equation}
	 	At this point, we can proceed as in Lemma \ref{minimalmass}: let $f_n:=\lambda(he^{u_n}-1)$ and $q>1$; then, if $q$ is close enough to $1$, we have the following estimate:
	 	\begin{equation}\label{est52}
	 		\begin{split}
	 			\int_{B_{r_i}(y_i)} e^{q(v_n-\bar{v}_n)}\dv(x)&\leq C\int_{B_{r_i}(y_i)}\exp\left(\int_{B_{2r_i}(y_i)}qG(x,y)f_n\dv(y)\right)\dv(x) \\
	 			&\leq C\sup_{y\in B_{2r_i}(y_i)}\int_{B_{r_i}(y_i)}e^{q\abs{G(x,y)}\normps{f_n}{1}{B_{2r_i}(y_i)}}\dv(x) \\
	 			&\leq C\sup_{y\in M} \int_M \left(\frac{1}{d_g(x,y)}\right)^{\frac{q\normps{f_n}{1}{B_{2r_i}(y_i)}}{2\pi}}\dv(x)<+\infty.
	 		\end{split}
	 	\end{equation}
	We can now argue as in \eqref{exponentialestimate} to recover that, for $p\in(1,q)$,
	\begin{equation*}
		\int_{B_{r_i}(y_i)}\abs{e^{u_n}-e^{v_n}}^p\dv=o_n(1) \qquad \forall i=1,\dots,A.
	\end{equation*}	 	
	 	From this estimate, \eqref{triangletrianglepap}, \eqref{est52} and $\bar{v}_n\to-\infty$ we get
	 	\begin{equation*}
	 		\begin{split}
	 			\int_{M\backslash B_{\frac{1}{2C}i(M)}}\lambda he^{u_n}\dv&=\int_{M\backslash B_{\frac{1}{2C}i(M)}}\lambda he^{v_n}\dv +o_n(1)=e^{\bar{v}_n}\int_{M\backslash B_{\frac{1}{2C}i(M)}}\lambda he^{(v_n-\bar{v}_n)}\dv+ o_n(1) \\
	 			&\leq C A e^{\bar{v}_n}\sup_{y\in M, \,\, i=1,\dots,A}\int_M \left(\frac{1}{d_g(x,y)}\right)^{\frac{\normps{f_n}{1}{B_{2r_i}(y_i)}}{2\pi}}\dv(x)+o_n(1) \\
	 			&\leq \tilde{C} e^{\bar{v}_n}+o_n(1)\xrightarrow{n\to\infty}0,
	 		\end{split}
	 	\end{equation*}
	 	which, together with \eqref{firsteq}, gives us the conclusion.
	 \end{proof}

	Finally, we consider the general case:
	
	\begin{step3*}
		If $j$ in Step 1 is arbitrary, then, up to a subsequence,
		\begin{equation}\label{step3quantization}
			\int_M \lambda he^{u_n}\dv=8\pi j+o_n(1).
		\end{equation}
	\end{step3*}
	
	\begin{proof}
		We need to proceed carefully as multiple bubbles might accumulate in the same point. The key idea is to analyze the different clusters and work inductively, following the same arguments of \cite{lishafrir-1994} and \cite{malchiodi-2006-CRELLE}.
		
		Without loss of generality, assume that
		\begin{equation}\label{step3eq}
			d_g(x_{1,n},x_{2,n})=\inf_{i\not=h} d_g(x_{i,n},x_{h,n}) \qquad \forall n\in \N.
		\end{equation}
		If $d_g(x_{1,n},x_{2,n})\not\to 0$, then each ``bubble" concentrates in a different point and we can argue as in Step 2 to get \eqref{step3quantization}.
		
		Assume now $d_g(x_{1,n},x_{2,n})\to0$ as $n\to\infty$. Let $X_n:=\{x_{1,n},\dots,x_{j,n}\}$ be the set of points of Step 1, and let $J_{1,n}\subseteq X_n$ be given by
		\begin{equation*}
			J_{1,n}:=\left\{x_{1,n},\dots,x_{\gamma,n}\mid \exists C>0 \text{ such that } d_g(x_{i,n},x_{1,n})\leq Cd_g(x_{1,n},x_{2,n}) \quad \forall i\not =1\right\}.
		\end{equation*}
		In other words, $J_{1,n}$ is the set of accumulation points for which the distance from $x_{1,n}$ is comparable to $d_g(x_{1,n},x_{2,n})$ when $n$ goes to infinity and, up to relabeling, we are assuming that $J_{1,n}$ is made up of the first $\gamma$ points of $X_n$, $\gamma\leq j$.
		By Step 1, the sequences $(x_{i,n})_n$, $(r_{i,n})_n$ and $(\frac{1}{C}d_g(x_{1,n},x_{2,n}))_n$ are simple blow-ups for $(u_n)_n$ for each $i=1,\dots,\gamma$. Hence we can apply Proposition \ref{simpleblowuptheorem} to see that, up to taking a larger $C$, one has
		\begin{equation}\label{star}
			\int_{B_{\frac{1}{C}d_g(x_{1,n},x_{2,n})}(x_{i,n})}\lambda he^{u_n}\dv=8\pi+o_n(1) \qquad \forall i=1,\dots,\gamma.
		\end{equation}
		
		To proceed further, we need to prove that there is no further accumulation of mass in a neighborhood of $J_{1,n}$ of size comparable to $d_g(x_{1,n},x_{2,n})$:
		
		\begin{lemma}\label{firstlemma}
			If $C>1$ is large enough, then
			\begin{equation*}
				\int_{B_{Cd_g(x_{1,n},x_{2,n})}(x_{1,n})} \lambda he^{u_n}\dv=8\pi \gamma+o_n(1).
			\end{equation*}
		\end{lemma}
		
		\begin{proof}
		To begin, we localise the problem as in Section \ref{sectionbubbleextr} by taking a sequence $(U_n,\psi_n)$ of local isothermal charts centered at $x_{1,n}$ and defining $\tilde{u}_n,\tilde{v}_n$ accordingly as in \eqref{localisedfunctions}. Let then $r_n:=d_g(x_{1,n},x_{2,n})$ and define the scaled functions $\hat{u}_n$ and $\hat{v}_n$ as in \eqref{localisedscaledfunctions}. By construction, we get that $r_n\psi_n(x_{i,n})\to \bar{x}_i\in\R^2$ $\forall i=1,\dots,\gamma$, and, by the scaling invariance of the mass, $\{\bar{x}_1\dots,\bar{x}_\gamma\}=:\mathcal{S}$ are the concentration points of $\hat{V}_n e^{\hat{u}_n}\,dx$.
		
		Let $C>1$ be so large that $\mathcal{S}\subset B_{\frac{C}{2}}(0)$. As in the proof of Proposition \ref{bubbleextractionthm}, let $\hat{w}_n:=\hat{v}_n+b_n$ with $b_n\in\R$ such that
		\begin{equation*}
			a_n:=\frac{1}{\abs{B_{2C}}}\int_{B_{2C}}\hat{u}_n\,dx=\frac{1}{\abs{B_{2C}}}\int_{B_{2C}}\hat{w}_n\,dx.
		\end{equation*}
		Then $\norm{\hat{u}_n-\hat{w}_n}_{H^1(B_{2C})}=o_n(1)$ as in \eqref{scaledapproximation} and $-\Delta\hat{w}_n=\hat{V}_ne^{\hat{u}_n}=:\hat{W}_n e^{\hat{w}_n}$, where $\hat{W}_n$ is uniformly bounded in $L^p(B_{C}) \,\, \forall p\in[1,+\infty)$ as can be seen by localising the Moser-Trudinger inequality.
		\begin{claim1*}\label{claim1}
			For any $\Omega\ssubset B_C(0)\backslash\mathcal{S}$, the sequence $(\hat{w}_n^+)_n$ is bounded in $L^\infty(\Omega)$.
		\end{claim1*}
		\begin{proof}
			By virtue of \eqref{P3}, for any fixed $\rho\in(0;2\pi)$ there exists $r=r(\rho)>0$, $r<\frac{1}{2}\textit{dist}(\Omega,\mathcal{S})$, such that $\forall x\in \Omega$ one has 
			\begin{equation}\label{massconditionclaim}
				\int_{B_r(x)}\hat{W}_n e^{\hat{w}_n}=\int_{B_r(x)}\hat{V}_n e^{\hat{u}_n}\leq \rho.
			\end{equation}
			
			The idea now is to apply \cite[Corollary 4]{brezismerle-1991} to $(\hat{w}_n)_n$ in $B_{\frac{r}{2}}(x)$ in order to conclude that $(\hat{w}_n^+)_n$ is bounded in $L^\infty(B_{\frac{r}{4}}(x))$; taking a finite cover of $\Omega$ we will then complete the proof of the claim.
			
			In order to use the aforementioned result, we need to check that $\exists \,\, p>1$ such that the following three assumptions hold:
			\begin{itemize}
				\item[(a)] $\norm{\hat{W}_n}_{L^p(B_{\frac{r}{2}}(x))}$ is uniformly bounded;
				\item[(b)] there exists $\epsilon>0$ such that $\int_{B_\frac{r}{2}(x)} \hat{W}_ne^{\hat{w}_n}\leq \epsilon<\frac{4\pi}{p'}$;
				\item[(c)] $\norm{\hat{w}_n^+}_{L^1(B_{\frac{r}{2}}(x))}$ is uniformly bounded.
			\end{itemize}
			Being the first two assumptions immediate consequences of Moser-Trudinger inequality \eqref{moser-trudinger-inequality} and of \eqref{massconditionclaim} respectively for a large enough $p>1$, it only remains to prove (c). Let $\psi_r$ be a smooth cutoff function such that $\psi_r\equiv1$ in $B_{\frac{r}{2}}(x)$ and $\psi_r\equiv0$ near $\partial B_r(x)$; let then
			\[
			\xi_n:=\psi_r \hat{w}_n+(1-\psi_r)a_n, \qquad \hat{\xi}_n:=\xi_n-a_n.
			\]
			Arguing exactly as in the proof of Proposition \ref{bubbleextractionthm}, we see that $\exists \, q>1$ such that $\int_{B_\frac{r}{2}(x)} e^{q\hat{\xi}_n}\leq C$, which in turn implies that $\int_{B_\frac{r}{2}(x)} e^{q\hat{w}_n}\leq C$ (recall that $a_n$ is uniformly bounded from above) and thus (c) holds.
		\end{proof}
		
		With the above Claim at hand, the rest of the proof now closely follows that of \cite[Theorem 3]{brezismerle-1991} apart from some minor adaptations.
		Let $\phi_n$ be the solution of 
		\begin{equation*}
			\begin{cases*}
				-\Delta \phi_n = \hat{W}_n e^{\hat{w}_n} & \text{in $B_C(0)$}, \\
				\phi_n =0 & \text{in $\partial B_C(0)$},
			\end{cases*}
		\end{equation*}
		and let $\eta_n:=\hat{w}_n-\phi_n$. Notice that each $\eta_n$ is harmonic and that, by Claim 1 and the maximum principle, one has $(\eta_n^+)_n\in L^\infty_{loc}(B_C\backslash \mathcal{S})$. As a consequence of Harnack's inequality, up to a subsequence:
		\begin{itemize}
			\item[(a)] either $(\eta_n)_n$ is bounded in $L^\infty_{loc}(B_C\backslash\mathcal{S})$;
			\item[(b)] or else $\eta_n\to-\infty$ uniformly on compact subsets of $B_C\backslash\mathcal{S}$.
		\end{itemize}
		\begin{claim2*}
			Alternative (a) cannot hold.
		\end{claim2*}
		\begin{proof}
			Let $r>0$ small enough so that $\mathcal{S}\cap B_{2r}(\bar{x}_1)=\{\bar{x}_1\}$. Assume by contradiction that (a) holds; then in particular $\eta_n$ is uniformly bounded in $L^\infty(\partial B_r(\bar{x}_1))$. Moreover, by Claim 1 we know that $\hat{W}_n e^{\hat{w}_n}$ is uniformly bounded in $L^p_{loc}(B_{2r}\backslash \{\bar{x}_1\})$ $\forall p>1$, so, by standard elliptic estimates, $\phi_n$ uniformly converges to a limit function $\phi$ in $L^\infty(\partial B_r(\bar{x}_1))$. Hence there exists a constant $C_0>0$ such that $\norm{\hat{w}_n}_{L^\infty(\partial B_r(\bar{x}_1))}\leq C_0$. Let $\zeta_n$ be the solution of
				\begin{equation*}
				\begin{cases*}
					-\Delta \zeta_n = \hat{W}_n e^{\hat{w}_n} & \text{in $B_r(\bar{x}_1)$}, \\
					\zeta_n =-C_0 & \text{in $\partial B_r(\bar{x}_1)$}.
				\end{cases*}
			\end{equation*}
			By the maximum principle, $\hat{w}_n\geq \zeta_n$ in $B_r(\bar{x}_1)$. Consider now a small value $0<\delta\ll r$; by Claim 1 we know that $\exists C=C_\delta>0$ such that, $\forall n\in\N$, $\hat{w}_n\leq C_\delta$ in $\{\delta<\abs{x-\bar{x}_1}<r\}=:A_{\delta,r}$. From this fact, arguing as in \eqref{exponentialestimate}, we obtain that 
			\[
			\int_{A_{\delta,r}}\abs{e^{\hat{w}_n}-e^{\hat{u}_n}}\,dx=o_n(1),
			\]
		    therefore
		    \begin{equation}\label{contrassumption}
		    \int_{A_{\delta,r}} e^{\zeta_n}\leq \int_{A_{\delta,r}} e^{\hat{w}_n}=\int_{A_{\delta,r}}e^{\hat{u}_n} +o_n(1)\leq \int_M e^{u_n}\dv+o_n(1)\leq \bar{C},
		    \end{equation}
		    where $\bar{C}$ \emph{does not} depend upon $\delta$.
		    
		    On the other hand, being $\hat{W}_n e^{\hat{w}_n}$ uniformly bounded in $L^1$, we see that  $\zeta_n$ is uniformly bounded in $W^{1,q}(B_r(\bar{x}_1))$ $\forall q\in[1;2)$ (see \cite{stampacchia1965egularity}) and also in any $L^p$, $p<+\infty$, by Sobolev embedding. Hence, up to a subsequence,
		    $\zeta_n$ a.e. converges to a limit function $\zeta$ (the convergence is also uniform on compact subsets of $B_r(\bar{x}_1)\backslash\{\bar{x}_1\}$) such that, if $\hat{W}_n e^{\hat{w}_n}\rightharpoonup \mu$ weakly in the sense of measures, then $-\Delta \zeta=\mu$ in $B_r(\bar{x}_1)$ with $\zeta=-C_0$ on $\partial B_r$.
		    Being $\mu\geq 0$ and $\mu\geq 8\pi\delta_{\bar{x}_1}$, by Green's representation formula we see that
		    \[
		    \zeta(x)\geq 4 \log\frac{1}{\abs{x-\bar{x}_1}} -C \qquad \text{for $x$ close to $\bar{x}_1$}.
		    \]
		    Therefore
		    \begin{equation*}
		    	\int_{A_{\delta,r}} e^{\zeta_n}\xrightarrow{n\to\infty}\int_{A_{\delta,r}} e^{\zeta}\xrightarrow{\delta\to 0}+\infty,
		    \end{equation*}
		    in contradiction with \eqref{contrassumption}.
		\end{proof}

		From Claim 2 and the uniform boundedness of $(\phi_n)_n$ on compact subsets of $B_C\backslash\mathcal{S}$, we see that $\hat{w}_n\to-\infty$ in $L^\infty_{loc}(B_C\backslash\mathcal{S})$ (the whole sequence!). Hence $\hat{W}_ne^{\hat{w}_n}\to 0$ in $L^p_{loc}(B_C\backslash\mathcal{S})$ and so, taking $\Omega\ssubset \R^2$ such that $\textit{dist}(\Omega,\mathcal{S})>0$ and $\psi_n(B_{Cd_g(x_{1,n},x_{2,n})}(x_{1,n})\backslash\bigcup_{i=1}^\gamma B_{\frac{1}{C}d_g(x_{1,n},x_{2,n})}(x_{i,n}))\subset \Omega$ for $n$ large enough (this is possible because the metric becomes close to the Euclidean one on shrinking balls), one gets
		\begin{equation*}
			\int_{B_{Cd_g(x_{1,n},x_{2,n})}(x_{1,n})\backslash\bigcup_{i=1}^\gamma B_{\frac{1}{C}d_g(x_{1,n},x_{2,n})}(x_{i,n})} \lambda he^{u_n}\dv\leq \int_\Omega \hat{V}_n e^{\hat{u}_n}\,dx=\int_\Omega \hat{W}_n e^{\hat{w}_n}\,dx \to0
		\end{equation*}
		as $n\to+\infty$, which, coupled with \eqref{star}, gives us the conclusion.
		\end{proof}

	At this point, let
	\begin{equation*}
		d_{1,n}:=\inf\left\{d_g(x_{1,n},x_{i,n})\mid x_{i,n}\in X_n\backslash J_{1,n}\right\}.
	\end{equation*}
	By definition, $\frac{d_{1,n}}{d_g(x_{1,n},x_{2,n})}\to +\infty$ as $n\to\infty$. We now show that there is no residual mass in between these two scales:
	
	\begin{lemma}\label{secondlemma}
		There exists a constant $C>0$ such that, up to a subsequence,
		\begin{equation*}
			\int_{B_{\frac{1}{C}d_{1,n}}(x_{1,n})}\lambda he^{u_n}\dv=8\pi \gamma +o_n(1).
		\end{equation*}
	\end{lemma}
	
	\begin{proof}
		The arguments are the same used to prove Proposition \ref{simpleblowuptheorem}. We know that $B_{C d_g(x_{1,n},x_{2,n})}(x_{1,n})$ carries a limit mass of $8\pi \gamma$ by Lemma \ref{firstlemma}. Moreover, as a consequence of \eqref{P3},
		\begin{equation}\label{analogueassumption}
			\begin{split}
			\forall \rho>0, \,\, \exists C_\rho>0 \quad \text{such that, if} &\quad \int_{B_s(y)}\lambda he^{u_n}\dv\geq \rho, \quad B_s(y)\subseteq B_{\frac{1}{C}d_{1,n}}(x_{1,n})\backslash B_{Cd_g(x_{1,n},x_{2,n})}(x_{1,n}), \\
			&\text{then} \quad s\geq\frac{1}{C_\rho} dist_g(y,J_{1,n}).
		\end{split}
		\end{equation}
		We can now argue as in Proposition \ref{simpleblowuptheorem} and assume by contradiction that $\exists \delta >0$ such that
		\begin{equation*}
			\liminf_{n\to\infty}\int_{B_{\frac{1}{C}d_{1,n}}(x_{1,n})\backslash B_{Cd_g(x_{1,n},x_{2,n})}(x_{1,n})}\lambda he^{u_n}\dv\geq\delta.
		\end{equation*}
		By virtue of \eqref{analogueassumption}, we have two alternatives which are the respective of \eqref{case1mass} and \eqref{case2assumption}. As a consequence, we have the same cases $(i)$ and $(ii)$ of Proposition \ref{simpleblowuptheorem}.
		
		In case $(i)$, we can scale and obtain an analogue of Lemma \ref{singularbubbleextraction}, namely there exist sequences $b_{n,m}$ of real numbers and $\tau_{n,m}$ with $0<\tau_{n,m}<d_{1,n} \,\, \forall n,m$, $\frac{d_g(x_{1,n},x_{2,n})}{\tau_{n,m}}\to 0$ $\forall m$ and a diagonal subsequence $n(m)$ such that the scaled sequence $\hat{w}_{n(m)}(x):=\tilde{v}_{n(m)}(\tau_{n(m),m} x)+2\log (\tau_{n(m),n})+b_{n(m),m}$ converges in $C^{\alpha}_{loc}(\R^2\backslash\{0\})$ to a limit function $\hat{w}$ satisfying
		\begin{equation*}
			\begin{cases*}
				-\Delta \hat{w}=\mu e^{\hat{w}} +8\pi(\gamma+\beta)\delta_0 & \text{in $\R^2$} \\
				\int_{\R^2}\mu e^{\hat{w}}\,dx<+\infty,
			\end{cases*}
		\end{equation*} 
		for some $\beta\in [0,1)$,
		which leads to a contradiction since, by Green's representation formula, one would get $e^{\hat{w}(x)}\geq\frac{e^{-C}}{\abs{x}^{4(\gamma+\beta)}}\notin L^1(B_1(0))$.
		
		In case $(ii)$ instead we can repeat verbatim the arguments in the proof of Proposition \ref{simpleblowuptheorem} (just use $Cd_g(x_{1,n},x_{2,n})$ and $d_{1,n}$ in place of $R_nr_n$ and $s_n$ respectively) in order to contradict \eqref{morseindexproperty}. This concludes the proof.
	\end{proof}
	
		\medskip
		
		$\mathit{End \,\, of  \,\,proof\,\,  of \,\,Step\,\, 3.}$ We can iteratively apply Lemma \ref{firstlemma} and Lemma \ref{secondlemma} to all clusters of points in order to reach our conclusion.
		Without loss of generality, assume to have only one point of concentration, namely that $x_{1,n},\dots, x_{j,n}$ all converge to the same point $x\in M$ as $n\to\infty$.
		Assume $j\geq 2 $ in Step 1 and apply Lemma \ref{firstlemma} a first time.
		
		If $J_{1,n}=X_n$, then, arguing as in Proposition \ref{simpleblowuptheorem} and Lemma \ref{secondlemma}, we can prove that
		\begin{equation*}
			\int_{ B_{\frac{1}{2C}i(M)}(x_{1,n})}\lambda he^{u_n}\dv=8\pi j +o_n(1). 
		\end{equation*}
		It is then sufficient to reason as in Step 2 to get \eqref{step3quantization}.
		
		If instead Card$(J_{1,n})< j$, then we can also apply Lemma \ref{secondlemma} and then consider
		\begin{equation*}
			J_{2,n}:=\left\{x\in X_n\backslash J_{1,n}\mid\exists C>0 \,\, \text{such that }d_g(x,x_{1,n})\leq C d_{1,n}\right\}\not=\emptyset,
		\end{equation*}
		where $C$ does not depend upon $n$. Write $J_{2,n}=:\{x_{\gamma+1,n},\dots,x_{\gamma',n}\}$ with $\gamma'\leq j$. By Proposition \ref{simpleblowuptheorem} and Lemma \ref{secondlemma}, there exists a constant $C>1$ such that
		\begin{equation*}
			\int_{B_{\frac{1}{C}d_{1,n}}(x_{l,n})}\lambda he^{u_n}\dv=\begin{cases*}
				8\pi \gamma + o_n(1) & \text{if $l=1$} \\
				8\pi + o_n(1) & \text{if $l=\gamma+1,\dots,\gamma'$.}
			\end{cases*}
		\end{equation*}
		We can then apply Lemma \ref{firstlemma} to such case (using $x_{\gamma+1,n}$ in place of $x_{1,n}$) and conclude that, if $C$ is large enough, then
		\begin{equation*}
			\int_{B_{C d_{1,n}}(x_{\gamma+1,n})} \lambda h e^{u_n}\dv=8\pi \gamma' + o_n(1).
		\end{equation*}
		Repeating this argument finitely many times we obtain \eqref{step3quantization}. This concludes the proof of Step 3.
	\end{proof}

	\begin{proof}[End of Proof of Theorem \ref{maintheorem}]
		The global quantization \eqref{step3quantization} above together with the normalization assumption \eqref{normalization} implies that $\lambda\in 8\pi\N$ whenever $(u_n)_n$ is unbounded in $H^1$-norm. The Theorem follows.
	\end{proof}

	 \section{The case of singular equations}\label{sectionsingulareq}
	 
	 In this last Section we will briefly explain up to what extent it is possible to generalise the above blow-up analysis to the case of Palais-Smale sequences for the functional $I_\lambda$ associated to the singular problem \eqref{singularequation}, that is,
	 \begin{equation}\label{singularfunctional}
	 	I_\lambda (u)= \frac{1}{2}\int_M \abs{\nabg u}^2 \dv +\lambda \int_M u\dv -\lambda \log\left(\int_M \tilde{h} e^u \dv\right),
	 \end{equation} 
	 where we recall that
	 \begin{equation}
	 	\tilde{h}(x):=\bar{h}(x)e^{-4\pi\sum_{i=1}^m\alpha_i G(x,q_i)}, 
	 \end{equation}
	 with $\bar{h}>0$ regular and $\alpha_i>-1$ $\forall i=1,\dots,m$. Define the following ``critical" set of parameters:
	 \begin{equation*}
	 	\Gamma_{\underline{\alpha}}:=\left\{8\pi n+\sum_{i\in I}8\pi(1+\alpha_i) \, \mid\,n\in \N, \,\, I\subseteq\{1,\dots,m\}, \,\,n+\text{Card}(I)\not= 0\right\}.
	 \end{equation*}
	  The main result is:
	 
	 \begin{thm}\label{singulartheorem}
	 	Given $\lambda\notin \Gamma_{\underline{\alpha}}$, let $(u_n)_n$ be a (PS)-sequence for \eqref{singularfunctional} satisfying the normalization condition \eqref{normalization} and the second order property \eqref{morseindexproperty}; assume further that $\alpha_i\in(-1,1]$ $\forall i=1,\dots,m$. Then $(u_n)_n$ is bounded in $H^1(M)$; in particular, it subconverges to a solution of \eqref{singularequation}.
	\end{thm}	
	 
	The additional assumption $\alpha_i\leq 1$ appears to be necessary in order to carry out the induction in Step 3 of Section \ref{sectionglobalquantization}. However, as already pointed out, we expect that it should be possible to overcome such limitation and thus prove Theorem \ref{singulartheorem} for any $\alpha_i>-1$, $i=1,\dots,m$, obtaining a compactness result for (PS)-sequences satisfying \eqref{morseindexproperty} identical to the one for sequences of exact solutions described in \cite{bartolucci-tarantello-2002-CMP}, \cite{bartolucci-montefusco-2007}.

	We now outline the main technical differences between the analysis of (PS)-sequences related to singular and regular Liouville functionals respectively.

	To begin, we only know in general that $0\leq\tilde{h}\in L^p(M)$ for some $p\in(1,+\infty]$ due to the eventual presence of poles. As in the regular case, if $(u_n)_n$ is bounded in $H^1(M)$, then there is nothing to prove. We thus assume by contradiction $(u_n)_n$ to be unbounded in $H^1(M)$, and define $(v_n)_n$ according to \eqref{vn}. Being $\normps{\tilde{h}e^{u_n}}{1}{M}=1$ $\forall n\in \N$, by weak compactness of measures we can find a nonnegative Radon measure $\nu$ and a subsequence along which $\lambda \tilde{h}e^{u_n}\rightharpoonup\nu$ weakly in the sense of measures. Define
	\begin{equation*}
		\mathscr{S}:=\left\{x\in M\mid \nu(x)\geq\frac{4\pi}{p'}\right\},
	\end{equation*}
	with $p'$ the Sobolev conjugate exponent of $p$: $\frac{1}{p}+\frac{1}{p'}=1$.
	\begin{claim*}
		Given $(u_n)_n$ as above, then $\mathscr{S}\not =\emptyset$.
	\end{claim*}
	 This can be proved by contradiction: if $\mathscr{S}=\emptyset$, then we can localise \cite[Corollary 4]{brezismerle-1991} and exploit the compactness of $M$ to obtain that $(v_n^+)_n$ is uniformly bounded in $L^\infty(M)$. Coupling this information with Lemma \ref{minimalmass} and the normalization condition in \eqref{vn} we can easily derive a contradiction.
	 
	 \bigskip

	 Once we know that the mass should concentrate somewhere, we can then proceed to isolate the blow-up profiles. Let $\bar{\alpha}:=\min\{0,\alpha_1,\dots,\alpha_m\}$ and, for $\rho\in (0, \pi(1+\bar{\alpha}))$, consider $(x_n)_n\subseteq M$, $(r_n)_n$ and $(\tilde{r}_n)_n$ satisfying \eqref{rhodef} and \eqref{scalingassumptions}.
	 Without loss of generality, assume that $x_n\to q\in M$ as $n\to+\infty$. If $q\notin\{q_1,\dots,q_m\}$, then we can apply Proposition \ref{bubbleextractionthm} to extract a bubble.
	 
	 Assume instead that $q=q_i$ for some $i\in\{1,\dots,m\}$. We then need to look at the sequence $\frac{d_g(x_n,q)}{r_n}$; we have two possibilities:
	 
	 	\underline{Case $(1)$.} Suppose that, up to subsequences, $\frac{d_g(x_n,q)}{r_n}\leq C$ for some positive constant $C$.
	 	
	 	In this case, pick $(U,\psi)$ isothermal chart centered at $q$ and such that $\psi(U)=B_{\delta}(0)\subseteq \R^2$ with conformal factor $e^\phi$. Let then $\eta$ and $\tilde{u}_n$, $\tilde{v}_n$ be defined as in \eqref{etadef}, \eqref{localisedfunctions} respectively (with $\psi$ in place of $\psi_n$).
	 We now scale those sequences as follows:
	 \begin{equation*}
	 	\begin{cases*}
	 		\hat{v}_n(x):=\tilde{v}_n(r_nx)+2(1+\alpha_i)\log r_n & \text{for $x\in B_{\frac{\delta}{r_n}}(0)$}, \\
	 		\hat{u}_n(x):=\tilde{u}_n(r_nx)+2(1+\alpha_i)\log r_n \\
	 		\hat{V}_n(x):=V(r_nx),
	 	\end{cases*}
	 \end{equation*}
	 so that
	 \begin{equation*}
	 	\begin{cases*}
	 		-\Delta \hat{v}_n=:\abs{x}^{2\alpha_i} \hat{V}_n(x)e^{\hat{u}_n(x)} & \text{in $B_\frac{\delta}{r_n}(0)$,} \\
	 		\int_{B_1(x)}\abs{x}^{2\alpha_i}\hat{V}_ne^{\hat{u}_n}\,dx<\pi(1+\bar{\alpha}) & \text{$\forall B_1(x)\subseteq B_{\frac{\delta}{r_n}}(0)$}.
	 	\end{cases*}
	 \end{equation*}
	 Let $R>1$ and define $(a_n)_n$, $(b_n)_n$, $(\hat{w}_n)_n$, $(\xi_n)_n$ and $(\hat{\xi}_n)_n$ as in \eqref{w_nhat}, \eqref{xi_ndef} respectively. The sequence $(a_n)_n$ is bounded from above, and we can argue as in the proof of Proposition \ref{bubbleextractionthm} to prove that \eqref{quadr} holds for some $q>1$. At this point, in order to estimate $(a_n)_n$ from below, we need to use the additional fact that $(\hat{w}_n^+)_n$ is bounded in $L^\infty(B_R)$, which follows, as in the proof of Lemma \ref{firstlemma}, from an application of \cite[Corollary 4]{brezismerle-1991}.
	 As soon as we have $\abs{a_n}\leq C$, we can argue as in Proposition \ref{bubbleextractionthm} to prove that, along a subsequence, one has $\hat{w}_n\to\hat{w}$ in $C^\alpha_{loc}(\R^2)\cap H^1_{loc}(\R^2)$, where $\hat{w}$ satisfies
	 \begin{equation*}
	 	\begin{cases*}
	 		-\Delta \hat{w}=\mu\abs{x}^{2\alpha_i}e^{\hat{w}} & \text{in $\R^2$}, \\
	 		\int_{\R^2}\mu\abs{x}^{2\alpha_i}e^{\hat{w}}\,dx=8\pi(1+\alpha_i),
	 	\end{cases*}
	 \end{equation*}
	 for some constant $\mu>0$. By the classification result of \cite{prajapattarantello}, we thus get a bubble of total mass $8\pi(1+\alpha_i)$.

	 \underline{Case $(2)$.} Suppose instead that, up to subsequences, $\frac{d_g(x_n,q)}{r_n}\xrightarrow{n\to\infty}+\infty$ and that the sequence is monotone increasing. Define $(U_n,\psi_n)$, $\phi_n$, $\eta_n$, $\tilde{u}_n$, $\tilde{v}_n$, $\tilde{V}_n$ as in Section \ref{sectionbubbleextr}. Let $(\gamma_n)_n\subseteq(0,1)$ be a sequence of numbers such that $\frac{d_g(q,x_n)}{r_n^{\gamma_n}}=1$ $\forall n\in \N$; in particular, $\gamma_n\to\gamma\in [0,1)$ since $(\gamma_n)_n$ is monotone decreasing. Define $y_n:=\psi_n(q)$ and consider the scaled sequences
	 \begin{equation*}
	 	\begin{cases*}
	 		\hat{v}_n(x):=\tilde{v}_n(r_nx)+2(1+\alpha_i\gamma_n)\log r_n & \text{for $x\in B_{\frac{\delta}{r_n}}(0)$}, \\
	 		\hat{u}_n(x):=\tilde{u}_n(r_nx)+2(1+\alpha_i\gamma_n)\log r_n. \\
	 	\end{cases*}
	 \end{equation*}
	 Write $\tilde{V}_n(x)=:\abs{x-y_n}^{2\alpha_i}V_n(x)$; then
	 \begin{equation*}
	 	-\Delta \hat{v}_n(x)=\hat{V}_n(x)e^{\hat{u}_n(x)}, \quad \text{where} \quad \hat{V}_n(x):=\abs{r_n^{1-\gamma_n}x+y_nr_n^{-\gamma_n}}^{2\alpha_i}V_n(r_nx).
	 \end{equation*}
	 Let $R>1$ fixed; for $n$ large enough, $\hat{V}_n$ is smooth and bounded from above and below by two positive constants in $B_R$. We can thus argue as in Proposition \ref{bubbleextractionthm} and isolate a concentration profile with limit mass of $8\pi$.
	 
	 \bigskip
	 
	 We now get to the delicate point of proving quantization in a simple blow-up. If $x_n\to x\notin\{q_1,\dots,q_m\}$, then everything works as in the ``regular" case, so, from now on, we'll assume that $x_n\to q_i$ for some $i\in\{1,\dots,m\}$. We migh have two different situations, related to the different alternatives above:
	 
	  \underline{Case $(1)$.} Let $(x_n)_n$, $(r_n)_n$ and $(s_n)_n$ be a simple blow-up for $(u_n)_n$ with central profile of mass $8\pi(1+\alpha_i)$: in other words, the three sequences satisfy Definition \ref{simpleblowupdefin} with \eqref{simpleblowupbubble} replaced by:
	 \begin{equation*}
	 	\exists R_n\to\infty \quad \text{such that} \,\, \frac{R_nr_n}{s_n}\to0 \,\,\, \text{and} \quad \norm{\hat{w}_n-\log\frac{\mu}{\left(1+\frac{1}{8(1+\alpha_i)}\abs{\cdot}^{2(1+\alpha_i)}\right)^2}}_{C^\alpha(B_{R_n})} \xrightarrow{n\to\infty}0.
	 \end{equation*}
	 We now want to prove the analogue version of Proposition \ref{simpleblowuptheorem} for this case, that is, we want to find a constant $C>1$ such that
	 \begin{equation*}
	 	\int_{B_\frac{s_n(x_n)}{C}} \lambda he^{u_n}\dv=8\pi(1+\alpha_i)+o_n(1).
	 \end{equation*}
	 This is indeed possible, and the proof follows the one of Proposition \ref{simpleblowuptheorem}: we argue by contradiction and assume \eqref{contradictionassumption} to hold; in this way we face the two alternatives of presence of residual mass in a dyadic ring \eqref{case1mass} or ``spreading" of residual mass \eqref{case2assumption}. In this last case, the proof proceeds exactly as in the regular case using suitable test functions to reach a contradiction with \eqref{morseindexproperty}. If instead we suppose \eqref{case1mass} to hold, then we can argue as in Lemma \ref{singularbubbleextraction} and, in the end, we obtain a function $\hat{w}$ which is a solution for
	 \begin{equation*}
	 	\begin{cases*}
	 		-\Delta \hat{w}=\mu\abs{x}^{2\alpha_i} e^{\hat{w}}+\beta \delta_0 &\text{in $\R^2$,} \\
	 		\int_{\R^2}\abs{x}^{2\alpha_i} e^{\hat{w}}\,dx<+\infty,
	 	\end{cases*}
	 \end{equation*}
	 for some $\mu>0$ and $\beta \in[8\pi(1+\alpha_i),\lambda)$. However, a quick computation with the Green representation formula readily shows that there exists $C>0$ such that
	 \begin{equation*}
	 	\hat{w}(x)\geq\frac{\beta}{2\pi}\log\frac{1}{\abs{x}}-C \qquad \text{for $x\in B_1(0)$},
	 \end{equation*}
	 therefore
	 \begin{equation*}
	 	\abs{x}^{2\alpha_i}e^{\hat{w}(x)}\geq\tilde{C}\frac{\abs{x}^{2\alpha_i}}{\abs{x}^{\frac{\beta}{2\pi}}}\geq\frac{\tilde{C}}{\abs{x}^{4+2\alpha_i}}\notin L^1(B_1(0)),
	 \end{equation*}
	 which is a contradiction.

	  \underline{Case $(2)$.} Assume that $(x_n)_n$, $(r_n)_n$ and $(s_n)_n$ are a simple blow-up for $(u_n)_n$ satisfying \eqref{simpleblowupbubble} (with central profile of mass $8\pi$) and with $x_n\to q_i$. As above, we may argue by contradiction and consider the same alternative between \eqref{case1mass} and \eqref{case2assumption}. Alternative \eqref{case2assumption} can be dealt with the usual argument, so we focus on \eqref{case1mass}, where we may consider $C_m\to+\infty$ and $\tau_{n,m}$ as in Section \ref{sectionsimpleblowup}. Consider now $\frac{\tau_{n,m}}{d_g(x_n,q_i)}$; up to subsequences, we might assume that
	  \begin{equation*}
	  	\exists \,\, \lim_{m\to\infty}\lim_{n\to\infty} \frac{\tau_{n,m}}{d_g(x_n,q_i)}=:L\in [0,+\infty].
	  \end{equation*}
	 Now, if $L<+\infty$, we can reason as in Lemma \ref{singularbubbleextraction} and reach a contradiction. However, if we assume instead that $L=+\infty$, then, arguing as in Lemma \ref{singularbubbleextraction}, we end up finding a scaled subsequence which converges to a solution of
	 \begin{equation*}
	 	\begin{cases*}
	 		-\Delta \hat{w}=\mu\abs{x}^{2\alpha_i} e^{\hat{w}}+\beta \delta_0 &\text{in $\R^2$,} \\
	 		\int_{\R^2}\abs{x}^{2\alpha_i} e^{\hat{w}}\,dx<+\infty,
	 	\end{cases*}
	 \end{equation*}
	 where this time $\beta\in[8\pi;\lambda)$. As a consequence, by virtue of Green's representation formula, we see that
	 \begin{equation*}
	 	\abs{x}^{2\alpha_i}e^{\hat{w}(x)}\geq\frac{\tilde{C}}{\abs{x}^{4-2\alpha_i}} \qquad \text{in $B_1(0)$},
	 \end{equation*}
	 so we reach a contradiction with the integrability condition \underline{only in case $\alpha_i\leq 1$}, which is the additional assumption considered in Theorem \ref{singulartheorem}.

	 \bigskip
	 
	 Finally, we can follow the arguments in Section \ref{sectionglobalquantization} with minor changes and reach the final conclusion that $\lambda\in \Gamma_{\underline{\alpha}}$ whenever $(u_n)_n$ is unbounded in $H^1$-norm, completing the proof of Theorem \ref{singulartheorem}.

	 \appendix
	 
	 \section{Proof of Lemma \ref{singularbubbleextraction}}\label{sectionappendix}
	 
	 The proof closely follows that of Proposition \ref{bubbleextractionthm}, thus we will only highlight the main differences. Define $\hat{u}_{n,m}$, $\hat{v}_{n,m}$ as in \eqref{hat1,nfunctions}. For $R>1$ and $n,m$ large enough, let $a_{n,m}:=\frac{1}{\abs{B_{2R}}}\int_{B_{2R}}\hat{u}_{n,m}\,dx$, and
	 $\hat{w}_{n,m}:=\hat{v}_{n,m}+b_{n,m}$ be such that
	 \begin{equation*}
	 	\frac{1}{\abs{B_{2R}}}\int_{B_{2R}}\hat{w}_{n,m}\,dx=\frac{1}{\abs{B_{2R}}}\int_{B_{2R}}\hat{u}_{n,m}\,dx.
	 \end{equation*}
	 Define also
	 \begin{equation*}
	 	\xi_{n,m}:=\psi_R\hat{w}_{n,m}+(1-\psi_R)a_{n,m}, \qquad \hat{\xi}_{n,m}:=\xi_{n,m}-a_{n,m};
	 \end{equation*}
	 here $\psi_R$ is a smooth radial cutoff which is identically $1$ inside $B_R$ and vanishes outside $B_{2R}$. By assumption, this time we got \eqref{case1localmass} instead of \eqref{scalingassumptions}, that is, there exists $\nu_R>0$ such that $\forall \,\, r<\nu_R$ one has
	 \begin{equation*}
	 	\int_{B_r(x)}\hat{V}_{n,m}e^{\hat{u}_{n,m}}\,dx<2\pi \qquad \forall x\in B_R\backslash B_{\frac{1}{R}}.
	 \end{equation*}
	  Moreover, as in Proposition \ref{bubbleextractionthm}, we can prove the validity of \eqref{triangolorovesciato} in our setting. As a consequence, we are in position to apply Remark \ref{localminmasslemma} to get
	 \begin{equation*}
	 	\int_{B_{R}\backslash B_{\frac{1}{R}}} e^{q\hat{\xi}_{n,m}}\leq C \qquad \text{uniformly in $n,m\in \N$ for some $q>1$.}
	 \end{equation*}
	 Being $a_{n,m}\leq\frac{1}{\abs{B_{2R}}}\int_{B_{2R}}e^{\hat{u}_{n,m}}\leq C$, this further implies
	 \begin{equation*}
	 	\int_{B_{R}\backslash B_{\frac{1}{R}}} e^{q\hat{w}_{n,m}}\leq C \qquad \text{uniformly in $n,m\in \N$ for some $q>1$.}
	 \end{equation*}
	 Notice also that $\norm{\hat{w}_{n,m}-\hat{u}_{n,m}}_{H^1(B_{2R})}=o_n(1)$. Hence, given $p\in [1,q)$, we can argue as in \eqref{exponentialestimate} to obtain
	 \begin{equation}\label{T3}
	 	\int_{B_{R}\backslash B_{\frac{1}{R}}}\abs{e^{p\hat{u}_{n,m}}-e^{p\hat{w}_{n,m}}}\,dx=o_n(1).
	 \end{equation}
	 Then
	 \begin{equation*}
	 		C e^{a_{n,m}}\geq e^{a_{n,m}}\int_{B_{R}\backslash B_{\frac{1}{R}}} e^{\hat{\xi}_{n,m}}\geq \frac{1}{C}\int_{\psi_n^{-1}(B_{R\tau_{n,m}}\backslash B_{\frac{\tau_{n,m}}{R}})} e^{u_n}\dv +o_n(1)\geq \frac{1}{C}\delta_{C_m}',
	 \end{equation*}
	 where $\delta_{C_m}'$ and the last inequality follow from \eqref{case1mass}. 
	 Therefore $\abs{a_{n,m}}\leq C(m)$, which, together with \eqref{T3} and standard elliptic estimates, leads to
	 \begin{equation*}
	 	\norm{\hat{w}_{n,m}}_{W^{2,p}(B_{\frac{R}{2}}\backslash B_{\frac{2}{R}})}\leq C\left(\norm{\hat{w}_{n,m}}_{L^p(B_R\backslash B_{\frac{1}{R}})}+\norm{\hat{V}_{n,m}e^{\hat{u}_{n,m}}}_{L^p(B_R\backslash B_{\frac{1}{R}})}\right)\leq C(m).
	 \end{equation*}
	  As a consequence, we can use a diagonal argument to find a subsequence $n(m)\to+\infty$  such that $\hat{w}_{n(m),m}$ strongly subconverges in $C^{0,\alpha}_{loc}(\R^2\backslash\{0\})\cap H^1_{loc}(\R^2\backslash \{0\})$ to a limit function $\hat{w}$.
	 
	 Finally, we show that $\hat{w}$ solves \eqref{singularlimitequation} in the sense of distributions. Multiply the equation solved by $\hat{w}_{n(m),m}$ for a test function $\varphi$ and integrate to get, for $n$ large enough,
	 \begin{equation}\label{TT}
	 	\int_{\R^2} \hat{w}_{n(m),m}(-\Delta \varphi) \,dx =\int_{\R^2} \hat{V}_{n(m),m}e^{\hat{u}_{n(m),m}}\varphi\,dx.
	 \end{equation}
	 By virtue of Proposition \ref{bubbleextractionthm}, we know that $\int_{B_\frac{R_{n(m)}r_{n(m)}}{\tau_{n(m),m}}}\hat{V}_{n(m),m} e^{\hat{u}_{n(m),m}}\to 8\pi$ as $m\to\infty$; moreover, being \eqref{T3} true $\forall R>1$, we see that, for $n$ large enough, 
	 	\begin{equation*}
	 		\int_{\R^2}\hat{V}_{n(m),m}e^{\hat{u}_{n(m),m}}\varphi\,dx\xrightarrow{m\to\infty}\int_{\R^2}\mu e^{\hat{w}}\varphi\,dx + \beta\varphi(0),
	 	\end{equation*}
	 	for some $\beta\in [8\pi,\lambda)$.
	 Here we also used the fact that $\hat{V}_{n(m),m}e^{\hat{u}_{n(m),m}}$ is uniformly bounded in $L^1$. This same fact also implies that $\hat{w}_{n(m),m}$ is equibounded in $W^{1,q}(B_1)$ $\forall q\in[1,2)$, see \cite{stampacchia1965egularity}, and, by Sobolev embedding, also in $L^p(B_1)$ $\forall p\geq 1$. But the a.e. convergence of $\hat{w}_{n(m),m}$ and the $L^p$ bound give $\hat{w}_{n(m),m}\to \hat{w}$ in $L^1(B_1)$; as a consequence, we can also pass to the limit in the left hand side of \eqref{TT} and recover that $\hat{w}$ is a distributional solution for \eqref{singularlimitequation}.

	 \bigskip 
	 
	 \begin{acknowledgements}
	 	We are grateful to Andrea Malchiodi for suggesting the problem, for his constant support and for many useful ideas and key observations. Also, we have benefited from the valuable comments of Daniele Bartolucci, Mattia Freguglia and Gabriella Tarantello. The author is a member of GNAMPA, as part of INdAM.
	 \end{acknowledgements}

	 \bibliography{bibliography-Liouville-equations}
	 \bibliographystyle{alpha}

\end{document}